\documentclass[journal]{IEEEtran}

\makeatletter
\def\ps@headings{%
\def\@oddhead{\mbox{}\scriptsize\rightmark \hfil \thepage}%
\def\@evenhead{\scriptsize\thepage \hfil \leftmark\mbox{}}%
\def\@oddfoot{}%
\def\@evenfoot{}}
\makeatother
\usepackage[pdftex]{graphicx}
\usepackage{amsmath}
\usepackage{amssymb}
\usepackage{array}
\usepackage[svgnames]{xcolor} 
\usepackage{mdwmath}
\usepackage{subfigure}
 \usepackage{bm}
\usepackage{tipa}
 \usepackage{hyperref}
 \usepackage{framed}
 \usepackage{booktabs}
 \usepackage{caption}
 \usepackage{url}

\usepackage{algorithm}
\usepackage{algorithmic}
\usepackage{epsfig}
\usepackage{overpic}
\usepackage{epstopdf}
\usepackage{amsfonts,setspace,amsthm,braket}

\graphicspath{{../pdf/}{../jpeg/}}
\DeclareGraphicsExtensions{.pdf,.jpeg,.png}
\DeclareMathOperator*{\argmin}{arg\,min}
\DeclareMathOperator*{\argmax}{arg\,max}
\usepackage{dsfont}
\newtheorem{theorem}{Theorem}
\newtheorem{lemma}{Lemma}
\newtheorem{proposition}[theorem]{Proposition}

\newtheorem{definition}{Definition}

\allowdisplaybreaks

\newcommand{\beqa}{\begin{eqnarray}}
\newcommand{\eeqa}{\end{eqnarray}}

\floatname{algorithm}{Algorithm}
\IEEEoverridecommandlockouts

\begin{document}
\title{Optimal Control for Generalized Network-Flow Problems} 


\author{
\IEEEauthorblockN{Abhishek Sinha, Eytan Modiano\\}
\IEEEauthorblockA{Laboratory for Information and Decision Systems, Massachusetts Institute of Technology, Cambridge, MA 02139}
Email: sinhaa@mit.edu,
modiano@mit.edu


}
\maketitle
\begin{abstract}
We consider the  problem of throughput-optimal packet dissemination, in the presence of an arbitrary mix of unicast, broadcast, multicast and anycast traffic, in a general wireless network. We propose an online dynamic policy, called Universal Max-Weight (UMW), which solves the above problem efficiently. To the best of our knowledge, UMW is the first throughput-optimal algorithm of such versatility in the context of generalized network flow problems.  Conceptually, the UMW policy is derived by relaxing the precedence constraints associated with multi-hop routing, and then solving a min-cost routing and max-weight scheduling problem on a  \emph{virtual network of queues}.
When specialized to the unicast setting, the  UMW policy yields a throughput-optimal \emph{cycle-free routing} and link scheduling policy.   This is in contrast to the well-known throughput-optimal \emph{Back-Pressure} (BP) policy which allows for packet cycling, resulting in excessive latency.  Extensive simulation results show that the proposed UMW policy incurs a substantially smaller delay as compared to the BP policy. The proof of throughput-optimality of the UMW policy combines ideas from stochastic Lyapunov theory with a sample path argument from \emph{adversarial queueing theory} and may be of independent theoretical interest. 
\end{abstract}
\section{Introduction} \label{intro}
The \emph{Generalized Network Flow} problem involves efficient transportation of messages, generated at source node(s), to a set of designated destination node(s) over a multi-hop network. Depending on the number of destination nodes associated with each source node, the problem is known either as \emph{unicast} (single destination node), \emph{broadcast} (all node are destination nodes), \emph{multicast} (some nodes are destination nodes) or \emph{anycast} (several choices for a single destination node). 
Over the last few decades, a tremendous amount of research effort has been directed to address each of the above problems in different networking contexts. However, despite the increasingly diverse mix of internet traffic, to the best of our knowledge, there exists no universal solution to the general problem, only isolated solutions that do not interoperate, and are often suboptimal.  In this paper, we provide the first such universal solution:  A throughput optimal dynamic control policy for the generalized network flow problem.

We start with a brief discussion of   the above networking problems and then survey the relevant literature. 

In the \textbf{Broadcast}  problem, packets generated at a source need to be distributed among all nodes in the network. In the classic paper of Edmonds \cite{edmonds}, the broadcast capacity of a wired network is derived and an algorithm is proposed to compute the maximum number of edge-disjoint spanning trees, which together achieve the maximum broadcast throughput. The  algorithm in \cite{edmonds}  is combinatorial in nature and does not have a wireless counterpart, with associated interference-free edge activations. Following Edmonds' work, a variety of different broadcast algorithms have been proposed in the literature, each one targeted to optimize different metrics such as delay \cite{czumaj2003broadcasting}, energy consumption \cite{widmer2005low} and fault-tolerance \cite{kranakis2001fault}. In the context of optimizing throughput,  \cite{massoulie2007randomized} proposes a randomized broadcast policy, which is optimal for wired networks. However, extending this algorithm to the wireless setting proves to be difficult \cite{towsley2008rate}. The authors of  \cite{Sinha:2016:TMB:2942358.2942390} propose an optimal broadcast algorithm for a general wireless network, albeit with exponential complexity. In a recent series of papers            \cite{sinha2015throughput} \cite{Sinha:2016:TBW:2942358.2942389}, a simple throughput-optimal broadcast algorithm has been proposed for wireless networks with an underlying DAG topology. However, this algorithm does not extend to non-DAG networks. 

The \textbf{Multicast}  problem is a generalization of the broadcast problem, in which the packets generated at a source node needs to be efficiently distributed to a subset of nodes in the network. In its combinatorial version, the multicast problem reduces to finding the maximum number of edge-disjoint trees, spanning the source node and destination nodes. This problem is known as the \emph{Steiner Tree Packing} problem, which is NP-hard \cite{jain2003packing}. Numerous algorithms have been proposed in the literature for solving the multicast problem. In  \cite{swati} \cite{bui2008optimal},  back-pressure type algorithms are proposed for multicasting over wired and wireless networks respectively. These algorithms forward packets over a set of pre-computed distribution trees, and are limited to the throughput obtainable by these trees. Moreover, computing and  maintaining these trees is impractical in  large and  time-varying networks.  We note that because of the need for packet duplications, the Multicast and Broadcast problems do not satisfy standard flow conservation constraints,  and thus the design of throughput-optimal algorithms is non-trivial.

The \textbf{Unicast} problem involves a single source and a single destination. The celebrated Back-Pressure (BP) algorithm \cite{tassiulas} was proposed for the unicast problem. In this algorithm, the routing and scheduling decisions are taken based on local queue length differences. As a result, BP explores all possible paths for routing and usually takes  a long time for convergence, resulting in considerable latency, especially in lightly loaded networks. Subsequently, a number of refinements have been proposed to improve the delay characteristics of the BP algorithm. In \cite{ying2011combining} BP is combined with hop-length based shortest path routing for faster route discovery, and 
 \cite{zargham2013accelerated} proposes a second order algorithm using the Hessian matrix to improve delay.

The \textbf{Anycast} problem involves routing from a single source to \emph{any one} of the several given destinations. Anycast is increasingly used in Content-Distribution Networks (CDNs) with geo-replicated contents  
\cite{sinha_CDN}.  
However, despite its immense commercial importance, to the best of our knowledge, no throughput-optimal algorithm is known for this problem. 

Our proposed solution uses a \emph{virtual network of queues} - one virtual queue per link in the network.  We solve the routing problem dynamically using a simple ``weighted-shortest-route" computation  on the virtual network and using the corresponding route on the physical network.  Optimal link scheduling is performed by a max-weight computation, also in the virtual network, and then using the resulting activation in the physical network.  The overall algorithm is dynamic, cycle-free, and solves the generalized routing and scheduling problem optimally (i.e., maximally stable or throughput optimal).  In addition to this, the proposed \textbf{UMW} policy has the following advantages:
 
\begin{enumerate}
\item \textbf{Generalized Solution:} Unlike the BP policy, which solves only the unicast problem, the proposed \textbf{UMW} policy efficiently addresses all of the aforementioned network flow problems in both wired and wireless networks in a very general setting.  
\item \textbf{Delay Reduction:} Although the celebrated BP policy is throughput-optimal,  its average delay performance is known to be poor due to occurrence of packet-cycling in the network \cite{ying2011combining} \cite{bui2009novel}. In our proposed \textbf{UMW} policy, each packet traverses  a dynamically selected \emph{acyclic} route, which drastically reduces the average latency. 
\item \textbf{State-Complexity Reduction:} Unlike the BP policy, which maintains \emph{per-flow} queues at each node, the proposed \textbf{UMW} policy maintains only a virtual-queue counter and a priority queue per link, irrespective of the number and type of flows in the network. This reduces the amount of overhead that needs to be maintained for efficient operation. 
\item \textbf{Efficient Implementation:} In the BP policy, routing decisions are made hop-by-hop by the intermediate nodes. This puts a considerable amount of computational overhead on the individual nodes. In contrast, in the proposed \textbf{UMW} policy, the entire route of the packets is determined at the source (similar to \emph{dynamic source routing} \cite{johnson1996dynamic}).  Hence, the entire computational requirement is transferred to the source, which often has higher computational/energy resources than the nodes in the rest of the network (\emph{e.g.}, wireless sensor networks). \end{enumerate}
The rest of the paper is organized as follows: In section \ref{system-model-section} we discuss the basic system model and formulate the problem. In section \ref{overview} we give a brief overview of the proposed \textbf{UMW} policy. In section \ref{virtual_queues} we discuss the structure and dynamics of the virtual queues, on which \textbf{UMW} is based. In section \ref{physical_queues} we prove its stability property in the multi-hop physical network. In section \ref{implementation-section} we discuss implementation details. In section \ref{simulation-section} we provide extensive simulation results, comparing \textbf{UMW} with other competing algorithms. In section \ref{conclusion-section} we conclude the paper with a few directions for further research. 

\section{System Model and Problem Formulation} \label{system-model-section}
\subsection{Network Model}
We consider a wireless network with arbitrary topology, represented by the graph $\mathcal{G}(V,E)$. The network consists of $|V|=n$ nodes and $|E|=m$ links. Time is slotted. A link, if activated, can transmit one packet per slot. Due to  wireless interference constraints, only certain subsets of links may be activated together at any slot. The set of all admissible link activations is known as the \emph{activation set} and is denoted by $\mathcal{M} \subseteq 2^E$. We do not impose any restriction on the structure of the activation set $\mathcal{M}$. As an example, in the case of \emph{node-exclusive} or primary interference constraint \cite{joo2009greedy}, the activation set $\mathcal{M}_{\text{primary}}$ consists of the set of all \emph{matchings} \cite{west2001introduction} in the graph $\mathcal{G}(V,E)$. Wired networks are a special case of the above model, where the activation set $\mathcal{M}_{\text{wired}}=2^{E}$.  In other words, in wired networks, packets can be transmitted over all links simultaneously. See Figure \ref{network} for an example of a wireless network with primary interference constraints.

\begin{figure}[h!]
\centering
\subfigure[a wireless network]{
	\begin{overpic}[width=0.15\textwidth]{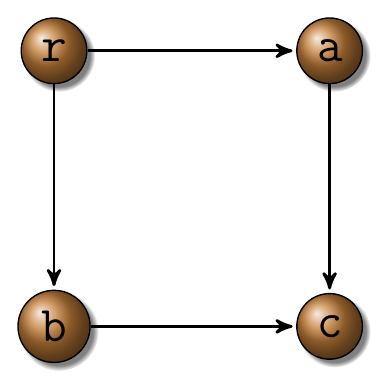}
	\end{overpic}
	\label{network-a}
}
\subfigure[activation vector $\bm{s}_1$]{
	\begin{overpic}[width=0.15\textwidth]{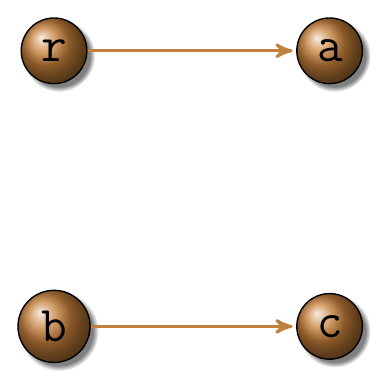}
	\end{overpic}
}
\subfigure[activation vector $\bm{s}_2$]{
  \begin{overpic}[width=0.15\textwidth]{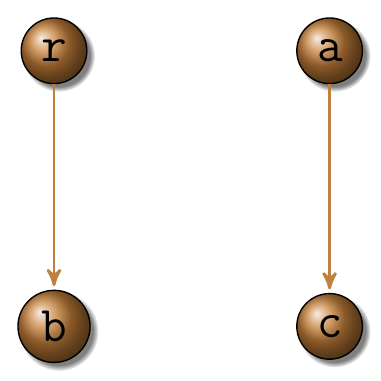}
  \end{overpic}
}
\caption{\small{A wireless network and its two maximal feasible link activations under the primary interference constraint.}}
\label{network}
\end{figure}

\subsection{Traffic Model}\label{traffic_model}
In this paper, we consider the \emph{Generalized Network Flow problem}, where incoming packets at a source node are to be distributed among an arbitrary set of destination nodes in a multi-hop fashion. Formally, the set of all distinct classes of incoming traffic is denoted by $\mathcal{C}$. A class $c$ traffic is identified by its source node $s^{(c)}\in V$ and the set of its required destination nodes $\mathcal{D}^{(c)} \subseteq V$. As explained below, by varying the structure of the destination set $\mathcal{D}^{(c)}$ of class $c$, this general framework yields the following four fundamental flow problems as special cases:
\begin{framed}
\begin{itemize}
	\item \textbf{\textsc{Unicast}}: All class $c$ packets, arriving at a source node $s^{(c)}$, are required to be delivered to a single destination node $\mathcal{D}^{(c)}=\{t^{(c)}\}$.
	\vspace{3pt}
	\item \textbf{\textsc{Broadcast}}: All class $c$ packets, arriving at a source node $s^{(c)}$, are required to be delivered to all nodes in the network, i.e., $\mathcal{D}^{(c)}=V$.
	\vspace{3pt}
	\item \textbf{\textsc{Multicast}}: All class $c$ packets, arriving at a source node $s^{(c)}$, are required to be delivered to a proper subset of nodes  $\mathcal{D}^{(c)}=\{t_1^{(c)}, t_2^{(c)}, \ldots, t_k^{(c)}\} \subsetneq V$. 
	\vspace{3pt}
	\item \textbf{\textsc{Anycast}}: A Packet of class $c$, arriving at a source node $s^{(c)}$, is required to be delivered to \emph{any one} of a given set of $k$ nodes $\mathcal{D}^{(c)}=t_1^{(c)} \oplus t_2^{(c)}\oplus \ldots \oplus t_k^{(c)}$. \\
	Thus the anycast problem is similar to the unicast problem, with all destinations forming a single \emph{super} destination node.\\
	\end{itemize} 
\end{framed}
Arrivals are i.i.d. at every slot, with $A^{(c)}(t)$ packets from class $c$ arriving at the source node $s^{(c)}$ at slot $t$. The mean rate of arrival for class $c$ is $\mathbb{E}A^{(c)}(t)=\lambda^{(c)}$. The arrival rate to the network is characterized by the vector $\bm{\lambda}=\{\lambda^{(c)}, c\in \mathcal{C}\}$. The total number of external packet arrivals to the entire network at any slot $t$ is assumed to be bounded by a finite number $A_{\max}$.

\subsection{Policy-Space}
 
An admissible policy $\pi$ for the generalized network flow problem executes the following two actions at every slot $t$: 
\begin{itemize}
\item \textsc{Link Activations}: Activating a subset of interference-free links $\bm{s}(t)$ from the activation set  $\mathcal{M}$.
\item \textsc{Packet Duplications and Forwarding}: Possibly duplicating \footnote{In order to transmit a packet over multiple downstream links (\emph{e.g.} in Broadcast or Multicast), the sender must duplicate the packet and send the copies to the respective downstream link buffers. } and forwarding packets over the activated links. Due to the link capacity constraint, at most one packet may be transmitted over an active link per slot.
\end{itemize}
The set of all admissible policies is denoted by $\Pi$. The set $\Pi$ is unconstrained otherwise and includes policies which may use all past and future packet arrival information.\\ 
A policy $\pi \in \Pi$ is said to \emph{support an arrival rate-vector $\bm{\lambda}$} if, under the action of the policy $\pi$, the destination nodes of any class $c$ receive distinct class $c$ packets at the rate $\lambda^{(c)}, c\in \mathcal{C}$. Formally, let $R^{(c)}(t)$ denote the number of distinct class-$c$ packets, received in common by all destination nodes $i \in \mathcal{D}^{(c)}$ \footnote{To be precise, the \emph{super}-destination node in case of Anycast.}, under the action of the policy $\pi$, up to time $t$.  
\begin{definition}{\emph{[Policy Supporting Rate-Vector $\bm{\lambda}$]}:} 
A policy $\pi \in \Pi$ is said to support an arrival rate vector $\bm{\lambda}$ if
\begin{eqnarray} \label{supporting_policy}
\liminf_{t \to \infty} \frac{R^{(c)}(t)}{t}= \lambda^{(c)}, \hspace{10pt} \forall c \in \mathcal{C}, \hspace{5pt}\mathrm{w.p. 1}
\end{eqnarray}
\end{definition}
The network-layer capacity region $\bm{\Lambda}(\mathcal{G}, \mathcal{C})$ \footnote{Note that, Network-layer capacity region is, in general (e.g. multicast), different from the Information-Theoretic capacity region \cite{ahlswede2000network}.} is defined to be the set of all supportable rates, i.e.,
\begin{eqnarray}\label{capacity-region}
\bm{\Lambda}(\mathcal{G}, \mathcal{C})\stackrel{\text{def}}{=} \{\bm{\lambda} \in \mathbb{R}^{|\mathcal{C}|}_+: \exists \pi \in \Pi \hspace{3pt}\textrm{supporting} \hspace{3pt}\bm{\lambda}\}
\end{eqnarray}
Clearly, the set $\bm{\Lambda}(\mathcal{G}, \mathcal{C})$ is \emph{convex} (using the usual \emph{time-sharing} argument). A policy $\pi^*\in \Pi$, which supports any arrival rate $\bm{\lambda}$ in the interior of the capacity region $\bm{\Lambda}(\mathcal{G}, \mathcal{C})$, is called a \emph{throughput-optimal} policy. 
\subsection{Admissible Routes of Packets}
 We will design a throughput-optimal policy, which delivers a packet $p$ to any node in the network \emph{at most} once.\footnote{This should be contrasted with the popular throughput-optimal unicast policy \emph{Back-Pressure} \cite{tassiulas}, which does not satisfy this constraint and may deliver the same packet to a node multiple times, thus potentially degrading its delay performance.} This immediately implies that the set of all admissible routes $\mathcal{T}^{(c)}$ for packets of any class $c$, in general, comprises of trees rooted at the corresponding source node $s^{(c)}$. In particular, depending on the type of class $c$ traffic, the topology of the admissible routes $\mathcal{T}^{(c)}$ takes the following special forms:
  \begin{framed}
\begin{itemize}
	\item \textbf{\textsc{Unicast Traffic}}: $\mathcal{T}^{(c)}=$ set of all $s^{(c)}-t^{(c)}$ paths in the graph $\mathcal{G}$. 
	\vspace{3pt}
\item \textbf{\textsc{Broadcast Traffic}}:  $\mathcal{T}^{(c)}=$ set of all spanning trees in the graph $\mathcal{G}$, rooted at  $s^{(c)}$. 
\vspace{3pt}
\item \textbf{\textsc{Multicast Traffic}}: $\mathcal{T}^{(c)}=$ set of all Steiner trees \cite{jain2003packing} in $\mathcal{G}$, rooted at $s^{(c)}$ and spanning the vertices $\mathcal{D}^{(c)}=\{t_1^{(c)}, t_2^{(c)}, \ldots, t_k^{(c)}\}$.
\vspace{3pt} 
\item \textbf{\textsc{Anycast Traffic}}: $\mathcal{T}^{(c)}=$ union of all $s^{(c)}-t_i^{(c)}$ paths in the graph $\mathcal{G}$, $i=1,2,\ldots, k$.  
\end{itemize}

\end{framed}

\subsection{Characterization of the Network-Layer Capacity Region}

Consider any arrival vector $\bm{\lambda}\in \bm{\Lambda}(\mathcal{G}, \mathcal{C})$. By definition, there exists an admissible policy $\pi \in \Pi$, which supports the arrival rate $\bm{\lambda}$ by means of storing, duplicating and forwarding packets efficiently. Taking time-averages over the actions of the policy $\pi$, it is clear that there exist a \emph{randomized} flow-decomposition and scheduling policy to route the packets such that none of the edges in the network is overloaded. Indeed, in the following theorem, we  show that for every $\bm{\lambda} \in \bm{\Lambda}(\mathcal{G}, \mathcal{C})$, there exist non-negative scalars $\{\lambda^{(c)}_i\}$, indexed by the admissible routes $T_i^{(c)}\in \mathcal{T}^{(c)}$ and a convex combination of the link activation vectors $\overline{\bm{\mu}} \in \text{conv}({\mathcal{M}})$ such that, 
\begin{eqnarray} 
\lambda^{(c)}= \sum_{T_i^{(c)} \in \mathcal{T}^{(c)}} \hspace{-5pt}\lambda^{(c)}_i, \hspace{10pt} \forall c \in \mathcal{C} \label{st1}\\
	\lambda_e \stackrel{(\text{def.})}{=} \hspace{-20pt}\sum_{(i,c): e \in T^{(c)}_i\hspace{-2pt}, T_i^{(c)} \hspace{-2pt}\in \mathcal{T}^{(c)}} \hspace{-25pt} \lambda^{(c)}_i \leq \overline{\mu}_e, \hspace{10pt} \forall e \in E.  \label{st2}
	\end{eqnarray}
 Eqn. \eqref{st1} denotes decomposition of the average incoming flows into different admissible routes and Eqn. \eqref{st2} denotes the fact that none of the edges in the network is overloaded, i.e. arrival rate of packets to any edge $e$ under the policy $\pi$ is \emph{at most} the rate allocated by the policy $\pi$ to the edge $e$ to serve packets.  \\
To state the result precisely, define the set $\overline{\bm{\Lambda}}$ to be the set of all arrival  vectors $\bm{\lambda} \in \mathbb{R}^{|\mathcal{C}|}_+$, for which there exists a randomized activation vector $\bm{\mu} \in \textrm{conv}(\mathcal{M})$ and a non-negative flow decomposition $\{\lambda_i^{(c)}\}$, such that Eqns. \eqref{st1} and \eqref{st2} are satisfied. We have the following theorem: 
\begin{framed}
\begin{theorem} \label{capacity-region-characterization}
The network-layer capacity region $\bm{\Lambda}(\mathcal{G}, \mathcal{C})$ is characterized by the set $\overline{\bm{\Lambda}}$, up to its boundary.  
\end{theorem}
\end{framed}
Proof of Theorem \ref{capacity-region-characterization} consists of two parts: converse and achievability. Proof of the \emph{converse} is given in Appendix \ref{capacity-region-characterization-converse-proof}, where we show that all supportable arrival rates must belong to the set $\overline{\bm{\Lambda}}$. The main result of this paper, as developed in the subsequent sections, 
is the construction of an efficient admissible policy, called \textbf{U}niversal \textbf{M}ax-\textbf{W}eight (\textbf{UMW}), which \emph{achieves} any arrival rate in the interior of the set $\overline{\bm{\Lambda}} $.\\


\section{Overview of the \textbf{UMW} Policy} \label{overview}
In this section, we present a brief overview of our throughput-optimal \textbf{UMW} policy, designed and analyzed in the subsequent sections. Central to the \textbf{UMW} policy is a global state vector, called virtual queues  $\bm{\tilde{Q}}(t)$, used for packet routing and link activations. Each component of the virtual queues is updated at every slot according to a one-hop queueing (Lindley) recursion, corresponding to a \emph{relaxed} network, described in detail in section \ref{virtual_queues}. Unlike the well-known Back-Pressure algorithm for the unicast problem \cite{tassiulas}, in which packet routing decisions are made hop-by-hop using physical queue-lengths $\bm{Q}(t)$, the \textbf{UMW} policy prescribes an admissible route to each incoming packet immediately upon its arrival (dynamic source routing). This route selection decision is dynamically made by solving a suitable min-cost routing problem (e.g., shortest path, MST etc.) at the source with edge costs given by the current virtual-queue vector $\tilde{\bm{Q}}(t)$. Link activation decisions at each slot are made by a Max-Weight algorithm with link-weights set equal to $\tilde{\bm{Q}}(t)$. Having fixed the routing and activation policy as above, in section \ref{physical_queues} we design a packet scheduling algorithm for the physical network, which efficiently resolves contention among multiple packets that wait to cross the same (active) edge at the same slot.  We show that the overall policy is throughput-optimal. One significantly new feature of our algorithm is that it is entirely oblivious to the length of the physical queues of the network and utilizes the auxiliary virtual-queue state variables  for stabilizing the former.\\
Our proof of throughput-optimality of \textbf{UMW} leverages ideas from \emph{deterministic} adversarial queueing theory and combines it effectively with the \emph{stochastic} Lyapunov-drift based techniques and may be of independent theoretical interest.

\section{Global Virtual Queues: Structures, Algorithms, and Stability} \label{virtual_queues}
\begin{figure}	
	\begin{overpic}[width=0.5\textwidth]{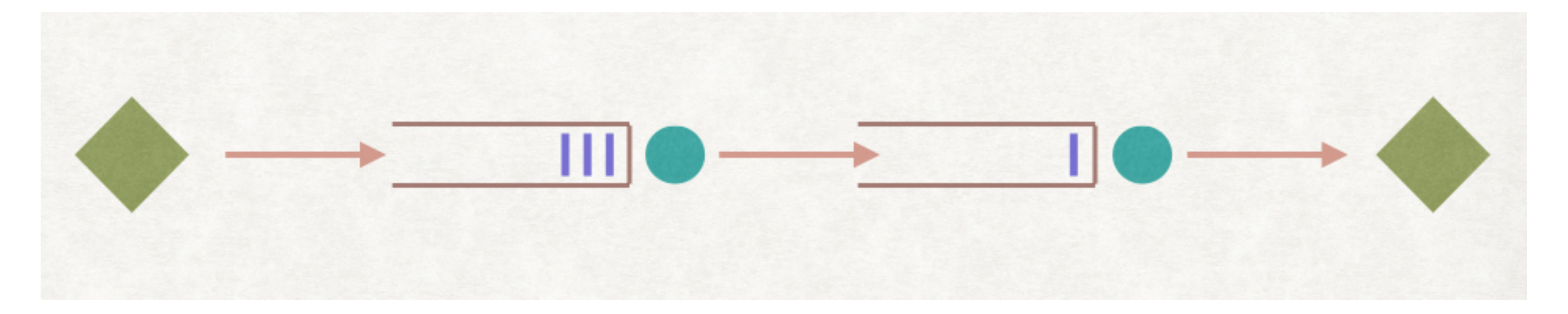}
		\put(12,12){\small{\textcolor{blue}{Arrival}}}
		\put(16,5){\small{$A(t)$}}
		\put(89,2.5){\small{sink}}
		\put(43,2){{\textcolor{blue}{departure}$|\textcolor{red}{Q_1(t)}$}}
		\put(24,14){$Q_1(t)$}
		\put(62,14){$Q_2(t)$}
	\end{overpic}
\caption{\small{A Multihop network with precedence constraints.}}
\label{tandem_queue}
\end{figure}
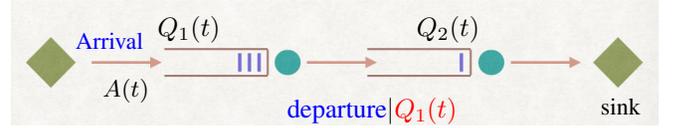
In this section, we introduce the notion of \emph{virtual queues} \footnote{Note that our notion of \emph{virtual-queues} is completely different from and unrelated to the notion of \emph{shadow-queues} proposed earlier in \cite{bui2009novel}, \cite{bui2008optimal} and \emph{virtual-queues} proposed in \cite{neely2006energy}. }, which is obtained by \emph{relaxing} the dynamics of the physical queues of the network  in the following intuitive fashion. 
\subsection{Precedence Constraints}
In a multi-hop network, if a packet $p$ is being routed along the path $T=l_1-l_2-\ldots-l_k$, where $l_i \in E$ is the $i$\textsuperscript{th} link on its path, then by the principle of causality, the packet $p$ cannot be physically transmitted over the $j$\textsuperscript{th} link $l_j$ if it has not already been transmitted by the first $j-1$ links $l_1,l_2, \ldots, l_{j-1}$. This constraint is known as the \emph{precedence constraint} in the network scheduling literature  \cite{lenstra1978complexity}. See Figure \ref{tandem_queue}. In the following, we make a radical departure by relaxing this constraint to obtain a simpler single-hop virtual system, which will play a key role in designing our policy and its optimality analysis. 
\subsection{The Virtual Queue Process $\{\tilde{\bm{Q}}(t)\}_{t \geq 1}$}
The \emph{Virtual queue process} $\tilde{\bm{Q}}(t)=\big(\tilde{Q}_e(t), e \in E\big)$  is an $|E|=m$ dimensional controlled stochastic process, imitating a fictitious queueing network \emph{without the precedence constraints}. In particular, when a packet $p$ of class $c$ arrives at the source node $s^{(c)}$, a dynamic policy $\pi$ prescribes a suitable route $T^{(c)}(t) \in \mathcal{T}^{(c)}$ to the packet. Denoting the set of all edges in the route $T^{(c)}(t)$ by  $\{l_1, l_2, \ldots, l_k\}$, this incoming packet induces a virtual arrival \emph{simultaneously} at each of the virtual queues $\big(\tilde{Q}_{l_i}\big), i=1,2,\ldots, k$, right upon its arrival to the source. Since the virtual network is assumed to be relaxed with no precedence constraints, any packet present in the virtual queue is eligible for service. See Figure \ref{virtual_queue} for an illustration. \\
\begin{figure}
\begin{overpic}[width=0.5\textwidth]{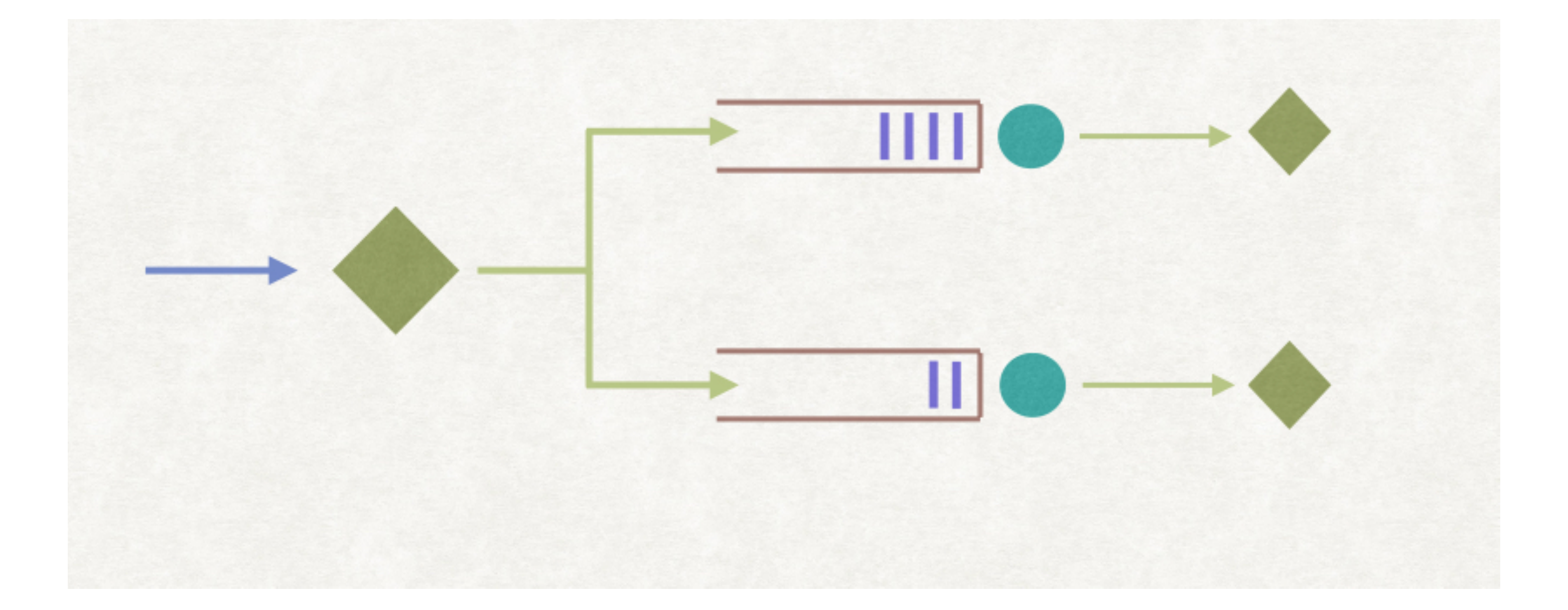}
\put(62,35){\textcolor{blue}{\small{$\mu_1(t)$}}}	
\put(62,19){\textcolor{blue}{\small{$\mu_2(t)$}}}
\put(7,24){\small{Arrival}}
\put(11,17){$\small{A(t)}$}
\put(37,34){\textcolor{blue}{$\small{A_1(t)}$}}
\put(37,9){\textcolor{blue}{$\small{A_2(t)}$}}
\put(50,23){$\tilde{Q}_1(t)$}
\put(50,7){$\tilde{Q}_2(t)$}
\put(21,13){\small{Source}}
\put(76,7){\small{virtual sinks }}
\end{overpic}
\caption{\small{Construction of the Virtual Queue process corresponding to the multi-hop network in Fig. \ref{tandem_queue}.}}
\label{virtual_queue}
\end{figure}
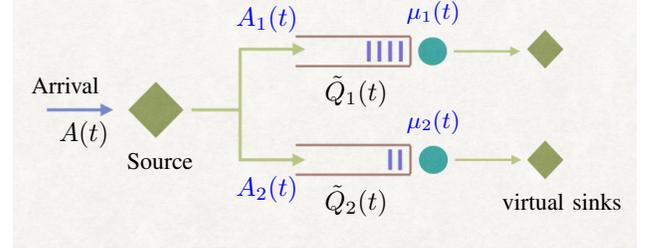
The (controlled) service process allocated to the virtual-queues is denoted by $\{\bm{\mu}^\pi(t)\}_{t \geq 1}$. We require the service process to satisfy the same activation constraints as in the original system, i.e., $\bm{\mu}^\pi(t) \in \mathcal{M}, \forall t \geq 1$. \\
Let $A^\pi_e(t)$ is the total number of virtual packet arrival (from all classes) to the queue $\tilde{Q}_e$ at time $t$ under the action of the policy $\pi$, i.e., 
 \begin{eqnarray} \label{routing_term}
A^\pi_e(t)= \sum_{c \in \mathcal{C}}A^{(c)}(t) \mathds{1}\big(e \in T^{(c)}(t)\big), \hspace{5pt} \forall e \in E. 
\end{eqnarray}

Hence, we have the following one-step evolution (Lindley recursion) of the virtual queue-process $\{\tilde{Q}_e(t)\}_{t \geq 1}$:

\begin{eqnarray} \label{queue-evolution}
\tilde{Q}_e(t+1)=\big(\tilde{Q}_e(t)+A^\pi_e(t)-\mu_e^\pi(t)\big)^+, \hspace{5pt} \forall e \in E, 
\end{eqnarray}

We emphasize that $A^\pi_e(t)$ is a function of the routing tree $T^{(c)}(t)$ that the policy chooses at time $t$, from the set of all admissible routes $\mathcal{T}^{(c)}$. This is discussed in the following. 

\subsection{Dynamic Control and Stability of the Virtual Queues} \label{dyn_contr}
Next we design a dynamic routing and link activation policy for the virtual network, which stabilizes the virtual queue process $\{\tilde{\bm{Q}}(t)\}_{t \geq 1}$, for all arrival rate-vectors  $\bm{\lambda}\in \mathrm{int}(\overline{\bm{\Lambda}})$. This policy is obtained by minimizing the one-step drift of a quadratic Lyapunov-function of the \emph{virtual queue-lengths} (as opposed to the real queue lengths used in the Back-Pressure policy \cite{tassiulas}). In the following section, we will show that when this dynamic policy is used in conjunction with a suitable packet scheduling policy in the physical network, the overall policy is throughput-optimal for the physical network.\\
To derive a stabilizing policy for the virtual network, consider a quadratic Lyapunov function $L(\tilde{\bm{Q}}(t))$ defined in terms of the virtual queue-lengths:
\begin{eqnarray*}
	L(\bm{\tilde{Q}}(t))= \sum_{e \in E} \tilde{Q}_e^2(t) 
\end{eqnarray*}
From the one-step dynamics of the virtual queues \eqref{queue-evolution}, we have:
\begin{eqnarray*}
\tilde{Q}_e(t+1)^2 &\leq& (\tilde{Q}_e(t)-\mu^\pi_e(t)+A^\pi_e(t))^2 \\
&=& \tilde{Q}_e^2(t) + (A_e^\pi(t))^2+(\mu_e^\pi(t))^2 +2\tilde{Q}_e(t)A^\pi_e(t)\\
&-& 2\tilde{Q}_e(t)\mu^\pi_e(t)-2\mu^\pi_e(t)A^\pi_e(t)
\end{eqnarray*}
Since $\mu^\pi_e(t) \geq 0$ and $A^\pi_e(t) \geq 0$, we have 
\begin{eqnarray*}
	\tilde{Q}_e^2(t+1) - \tilde{Q}_e^2(t) &\leq & (A_e^\pi(t))^2 + (\mu_e^\pi(t))^2 \\
	&+&2\tilde{Q}_e(t)A^\pi_e(t)-2\tilde{Q}_e(t)\mu^\pi_e(t)
\end{eqnarray*}
Hence, the one-step Lyapunov drift $\Delta^\pi(t)$, conditional on the current virtual queue-lengths $\bm{\tilde{Q}}(t)$, under the operation of any admissible Markovian policy $\pi \in \Pi$ is upper-bounded by 
\begin{eqnarray}\label{drift_expr}
	\Delta^\pi(t)&\stackrel{\mathrm{def}}{=}&\mathbb{E}\big(L(\bm{\tilde{Q}}(t+1))-L(\bm{\tilde{Q}}(t))|\bm{\tilde{Q}}(t)\big) \nonumber \\
	&\leq& B + 2 \sum_{e \in E}\tilde{Q}_e(t)\mathbb{E}\big(A_e^\pi(t)|\bm{\tilde{Q}}(t)\big)\nonumber \\
	&-& 2 \sum_{e \in E}\tilde{Q}_e(t)\mathbb{E}\big(\mu_e^\pi(t)|\bm{\tilde{Q}}(t)\big)
\end{eqnarray} 
where $B$ is a constant, bounded by $\sum_{e}(\mathbb{E}(A_e^\pi(t))^2+\mathbb{E}(\mu_e^\pi(t))^2)\leq A_{\max}^2+m$. \\
The upper-bound on the drift, given by \eqref{drift_expr}, holds good for any admissible policy in the virtual network. In particular, by minimizing the upper-bound point wise, and exploiting the separable nature of the objective, we derive the following decoupled dynamic routing and link activation policy for the virtual network:
\subsubsection*{Dynamic Routing Policy}
The optimal route for each class $c$,  over the set of all admissible routes, is selected by minimizing the following cost function, appearing in the middle of Eqn. \eqref{drift_expr}
\begin{eqnarray*}
	\mathsf{Routing Cost}^\pi\equiv \sum_{e\in E} \tilde{Q}_e(t) A_e^{\pi}(t),
\end{eqnarray*}
where we remind the reader that $A_e^\pi(t)$ are the routing policy dependent arrivals to the virtual-queue corresponding to the link $e$ at time $t$. \\
Using Eqn. \eqref{routing_term}, we may rewrite 
the objective-function as 
\begin{eqnarray}\label{w-r}
	\mathsf{Routing Cost}^\pi=\sum_{c \in \mathcal{C}} A^{(c)}(t) \bigg(\sum_{e \in E} \tilde{Q}_e(t) \mathds{1}\big(e \in T^{(c)}(t))\bigg)
\end{eqnarray}
Using the separability of the objective \eqref{w-r}, the above optimization problem decomposes into following min-cost route-selection problem $T_{\mathrm{opt}}^{(c)}(t)$ for each class $c$:
\begin{eqnarray}\label{opt_route}
	T_{\mathrm{opt}}^{(c)}(t) \in \argmin_{T^{(c)} \in \mathcal{T}^{(c)}} \bigg( \sum_{e \in E} \tilde{Q}_e(t) \mathds{1}\big(e \in T^{(c)}) \bigg)
\end{eqnarray}
Depending on the type of flow of class $c$, the optimal route-selection problem \eqref{opt_route} is equivalent to one of the following well-known combinatorial problems on the graph $\mathcal{G}$, with its edges weighted by the virtual queue-length vector $\bm{\tilde{Q}}$:
\begin{framed}
\begin{itemize}
	\item \textbf{\textsc{Unicast Traffic:}} $T_{\mathrm{opt}}^{(c)}(t)=$ The shortest $s^{(c)}-t^{(c)}$ path in the weighted-graph $\mathcal{G}$.
	\item \textbf{\textsc{Broadcast Traffic:}} $T_{\mathrm{opt}}^{(c)}(t)=$ The minimum weight spanning tree rooted at the source $s^{(c)}$, in the weighted-graph $\mathcal{G}$. \vspace{3pt}
	\item \textbf{\textsc{Multicast Traffic:}}
	$T_{\mathrm{opt}}^{(c)}(t)=$ The minimum weight Steiner tree rooted at the source $s^{(c)}$ and spanning the destinations $\mathcal{D}^{(c)}=\{t_1^{(c)}, t_2^{(c)}, \ldots, t_k^{(c)}\}$, in the weighted-graph $\mathcal{G}$.\vspace{3pt}
	\item \textbf{\textsc{Anycast Traffic:}} $T_{\mathrm{opt}}^{(c)}(t)=$ The shortest of the $k$ shortest $s^{(c)}-t_i^{(c)}$ paths, $i=1,2,\ldots, k$ in the weighted-graph $\mathcal{G}$.
\end{itemize}
\end{framed}
 Thus, the routes are selected according to a \emph{dynamic source routing} policy \cite{johnson1996dynamic}. Apart from the minimum weight Steiner tree problem for the multicast traffic (which is NP-hard with several known efficient approximation algorithms \cite{byrka2010improved}), all of the above routing problems on the \emph{weighted} virtual graph may be solved efficiently using standard algorithms \cite{cormen2009introduction}. \vspace{5pt}

\subsubsection*{Dynamic Link Activation Policy}
 A feasible link activation schedule $\bm{\mu}^*(t)\hspace{-1pt} \in \hspace{-1pt}\mathcal{M}$ is dynamically chosen at each slot by \emph{maximizing} the last term in the upper-bound of the drift-expression \eqref{drift_expr}, given as follows:
 \begin{eqnarray}\label{max-wt-sch}
 \bm{\mu}^*(t)  \in \argmax_{\bm{\mu} \in \mathcal{M}} \bigg(\sum_{e \in E} \tilde{Q}_e(t)\mu_e \bigg)
 \end{eqnarray}
 This is the well-known max-weight scheduling policy, which can be solved efficiently under various interference models (e.g., \emph{Primary} or node-exclusive model \cite{bui2009distributed}).\\
 In solving the above routing and scheduling problems, we tacitly made the assumption that the virtual queue vector $\bm{\tilde{Q}}(t)$ is \emph{globally known} at each slot. We will discuss practical distributed implementation of our algorithm in section \ref{implementation-section}.     \\
Next, we establish stability of the virtual queues under the above policy, which will be instrumental for proving throughput-optimality of the overall \textbf{UMW} policy:
 \begin{framed}
 \begin{theorem} \label{stability_theorem}
 	Under the above dynamic routing and link scheduling policy, the virtual-queue process $\{\bm{\tilde{Q}}(t)\}_{t\geq 0}$ is strongly stable for any arrival rate $\bm{\lambda} \in \mathrm{int}(\overline{{\bm{\Lambda}}})$, i.e., 
 	\begin{eqnarray*}
 		\limsup_{T\to \infty} \frac{1}{T}\sum_{t=0}^{T-1}\sum_{e \in E} \mathbb{E}(\tilde{Q}_e(t)) < \infty.
 	\end{eqnarray*}
 \end{theorem}
 \end{framed}
 \begin{proof}
 The proof involves a Lyapunov drift-argument. See Appendix \ref{stability_proof} for details. 
 \end{proof}
 As a consequence of the strong stability of the virtual-queues $\{\tilde{Q}_e(t), e \in E\}$, we have the following sample-path result, which will be the key to our subsequent analysis:
 \begin{framed}
 \begin{lemma} \label{rate_stability}
 Under the action of the above policy, we have for any  $\bm{\lambda} \in \mathrm{int}(\bar{\bm{\Lambda}})$:
 	\begin{eqnarray*}
 		\lim_{t \to \infty} \frac{\tilde{Q}_e(t)}{t}=0, \hspace{5pt}\forall e \in E,\hspace{10pt} \mathrm{w.p.}\hspace{3pt} 1.
 	\end{eqnarray*}
 In other words, the virtual queues are rate-stable \cite{neely2010stochastic}.	
 \end{lemma} 
 \end{framed}
 \begin{proof}
 	See Appendix \ref{rate_stability_proof}.
 \end{proof}
 The sample path result of Lemma \ref{rate_stability} may be interpreted as follows: For any given realization $\omega$ of the underlying sample space $\Omega$, define the function 
 \begin{eqnarray*}
 	F(\omega, t)= \max_{e \in E} \tilde{Q}_e(\omega, t). 
 \end{eqnarray*}
 Note that, for any $t \in \mathbb{Z}_+$, due to the boundedness of arrivals per slot, the function $F(\omega, t)$ is well-defined and finite.  
 In view of this, Lemma \eqref{rate_stability} states that under the action of the \textbf{UMW} policy, $F(\omega, t)=o(t)$ \emph{almost surely}. \footnote{$g(t)=o(t)$ if $\lim_{t\to \infty}\frac{g(t)}{t}=0$.} This result will be used in our sample pathwise  stability analysis of the physical queueing process $\{\bm{Q}(t)\}_{t\geq 0}$.

 \subsection{Consequence of the Stability of the Virtual Queues}
 It is apparent from the virtual-queue evolution equation \eqref{queue-evolution}, that the stability of the virtual queues under the \textbf{UMW} policy implies that the arrival rate at each virtual queue is \emph{at most} the service rate offered to it under the \textbf{UMW} routing and scheduling policy. In other words, \emph{effective load} of each edge $e$ in the virtual system is at most unity. This is a necessary condition for stability of the physical queues when the same routing and link activation policy is used for the multi-hop physical network. In the following, we make the notion of ``effective-load" mathematically precise. \paragraph*{Skorokhod Mapping}
 Iterating on the system equation \eqref{queue-evolution}, we obtain the following well-known discrete time Skorokhod-Map representation \cite{meyn2008control} of the virtual queue dynamics
 \begin{eqnarray} \label{skorokhod}
 	\tilde{Q}_e(t) = \bigg(\sup_{1\leq \tau \leq t} \big( A^\pi_e(t-\tau,t)-S^\pi_e(t-\tau, t)\big)\bigg)^+, 
 \end{eqnarray}
   where $A_e^\pi(t_1,t_2)\stackrel{\mathrm{def}}{=}\sum_{\tau=t_1}^{t_2-1} A^\pi_e(\tau)$, is the total number of arrivals to the virtual queue $\tilde{Q}_e$ in the time interval $[t_1,t_2)$ and $S_e^\pi(t_1,t_2)\stackrel{\mathrm{def}}{=}\sum_{\tau=t_1}^{t_2-1} \mu^\pi_e(\tau)$, is the total amount of service allocated to the virtual queue $\tilde{Q}_e$ in the interval $[t_1,t_2)$. For reference, we provide a proof of Eqn. \eqref{skorokhod} in Appendix \ref{skorokhod_proof}.  \\
   Combining Equation \eqref{skorokhod} with Lemma \ref{rate_stability}, we conclude that under the \textbf{UMW} policy, \emph{almost surely} for any sample path $\omega \in \Omega$, for each edge $e \in E$ and any $t_0< t$, we have
   \begin{eqnarray} \label{bounded_arrival}
   	A_e(\omega; t_0, t) \leq S_e(\omega; t_0,t) + F(\omega, t), 
   \end{eqnarray}
   where $F(\omega, t)=o(t)$. 
   \paragraph*{Implications for the Physical Network} Note that, every packet arrival to a virtual queue $\tilde{Q}_e$ at time $t$ corresponds to a packet in the physical network, that will eventually cross the edge $e$. Hence the loading condition \eqref{bounded_arrival} implies that under the \textbf{UMW} policy, the total number of packets injected during any time interval $(t_0,t]$, willing to cross the edge $e$, is less than the total amount of service allocated to the edge $e$ in that time interval up to an additive term of $o(t)$. Thus informally, the ``effective load" of any edge $e \in E$ is at most unity.  \\
  By utilizing the sample-path result in Eqn. \eqref{bounded_arrival}, in the following section we show that there exists a simple packet scheduling scheme for the physical network, which guarantees the stability of the physical queues, and consequently, throughput-optimality.

\section{Optimal Control of the Physical Network} \label{physical_queues}
With the help of the virtual queue structure as defined above, we next focus our attention on designing a throughput-optimal control policy for the multi-hop physical network. As discussed in Section \ref{system-model-section}, a control policy for the physical network consists of three components, namely (1) Routing, (2) Link activations and (3) Packet scheduling. In the proposed \textbf{UMW} policy, the (1) Routing and (2) Link activations for the physical network is done exactly in the same way as in the virtual network, based on the current values of the virtual queue state variables $\tilde{\bm{Q}}(t)$, described in Section \ref{dyn_contr}. There exist many possibilities for the third component, namely the packet scheduler, which efficiently resolves contention when multiple packets attempt to cross an active edge $e$ at the same time-slot $t$. 
Popular choices for the packet scheduler include FIFO, LIFO etc. 
In this paper, we focus on a particular scheduling policy which has its origin in the context of \emph{adversarial queueing theory} \cite{andrews2001universal}. In particular, we extend the \emph{Nearest To Origin} (NTO) policy to the generalized network flow setting, where a packet may be duplicated. This policy was proposed in \cite{Gamarnik} in the context of wired networks for the unicast problem. Our proposed scheduling policy is called Extended NTO (\textbf{ENTO}) and is defined as follows:

\begin{framed}
 \begin{definition}[Extended NTO]
  If multiple packets attempt to cross an active edge $e$ at the same time slot $t$, the Extended Nearest To Origin (\textbf{ENTO}) policy gives priority to the packet which has traversed the least number of hops along its path from its origin up to the edge $e$. 
 \end{definition}

\end{framed}

The Extended NTO policy may be easily implemented by maintaining a single priority queue \cite{cormen2009introduction} per edge. The initial priority of each incoming packet at the source is set to zero. Upon transmission by any edge, the priority of a transmitted packet is decreased by one. The transmitted packet is then copied into the next-hop priority queue(s) (if any) according to its assigned route. See Figure \ref{ENTO_fig} for an illustration. The pseudo code for the full \textbf{UMW} algorithm is provided in Algorithm \ref{UMW_algo}. 
 
 \begin{figure}
  \includegraphics[width=0.5\textwidth]{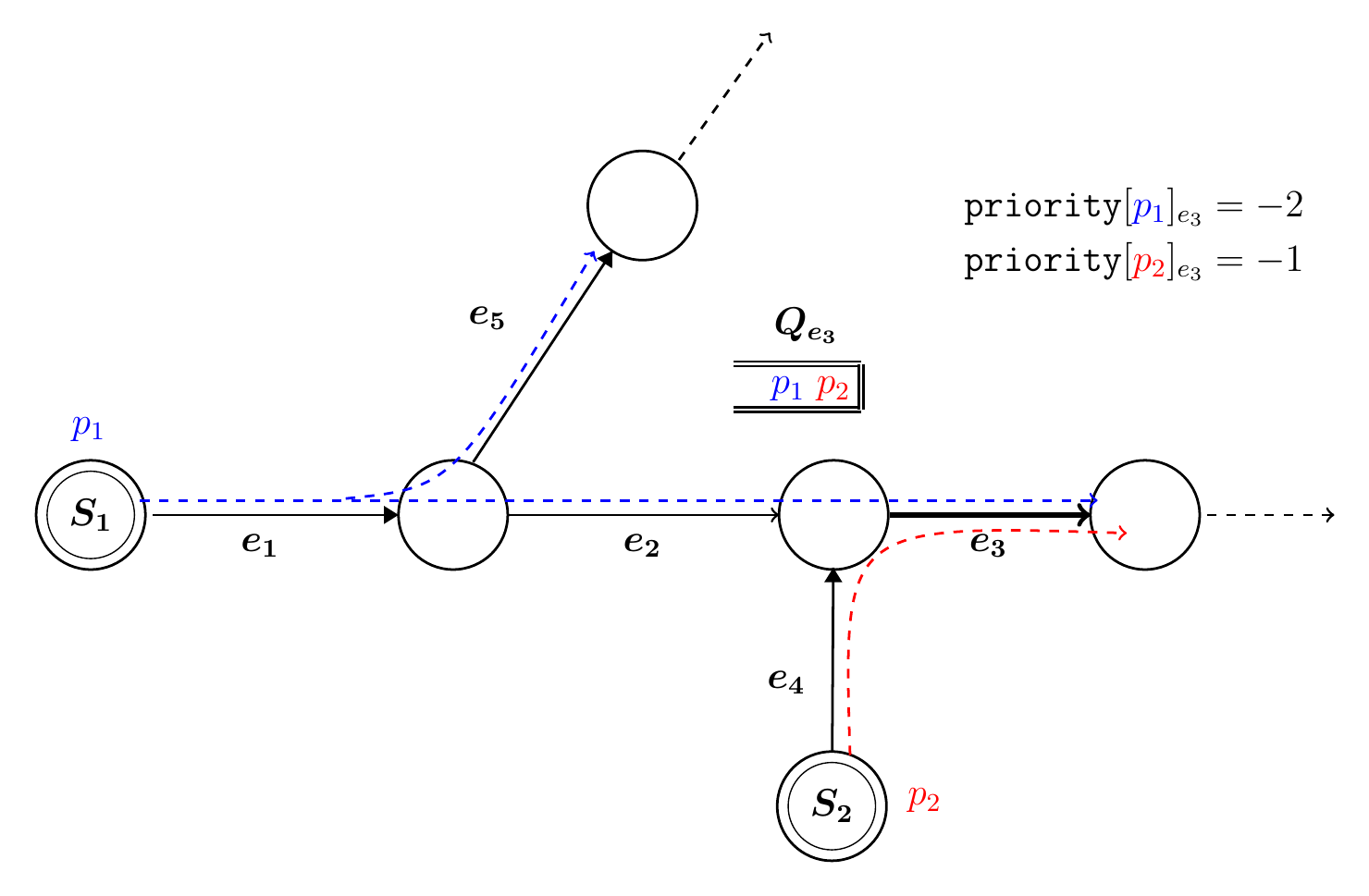}
  \caption{\small{A schematic diagram showing the scheduling policy ENTO in action. The packets $p_1$ and $p_2$ originate from the sources $S_1$ and $S_2$. Part of their assigned routes are shown in blue and red respectively. The packets contend for crossing the active edge $e_3$ at the same time slot. According to the ENTO policy, the packet $p_2$ has higher priority (having crossed a single edge $e_4$ from its source) than $p_1$ (having crossed two edges $e_1$ and $e_2$ from its source) for crossing the edge $e_3$. Note that, although a copy of $p_1$ might have already crossed the edge $e_5$, this edge does not fall in the path connecting the source $S_1$ to the edge $e_3$ and hence does not enter into priority calculations. }}
  \label{ENTO_fig}
 \end{figure}

 \begin{algorithm}  
\caption{Universal Max-Weight Algorithm (\textbf{UMW}) at slot $t$ for the Generalized Flow Problem in a Wireless Network}
\label{UMW_algo}
\begin{algorithmic}[1] 
 \REQUIRE Graph $\mathcal{G}(V,E)$, Virtual Queue-Lengths $\{\tilde{Q}_e(t), e \in E\}$ at the slot $t$.
 \STATE \textbf{[Edge-Weight Assignment]} Assign each edge of the graph $e \in E$ a weight $W_e(t)$ equal to $\tilde{Q}_e(t)$, i.e.
 \begin{eqnarray*}
 \bm{W}(t) \gets \bm{\tilde{Q}}(t)
 \end{eqnarray*}
 \STATE \textbf{[Route Assignment]} Compute a Minimum Weight Route $T^{(c)}(t) \in \mathcal{T}^{(c)}(t)$ for a class $c$ incoming packet in the weighted graph $\mathcal{G}(V,E)$, according to Eqn. \eqref{opt_route}.
 \STATE \textbf{[Link Activations]} Choose the activation $\bm{\mu}(t)$ from the set of all feasible activations $\mathcal{M}$, which maximizes the total activated link-weights, i.e.
 \begin{eqnarray*}
 \bm{\mu}(t) \gets \arg \max_{\bm{s} \in \mathcal{M}} \bm{s}\cdot \bm{W}(t)
 \end{eqnarray*}
\STATE \textbf{[Packet Forwarding]} Forward physical packets from the physical queues over the activated links according to the \textbf{ENTO} scheduling policy. 

\STATE \textbf{[Virtual Queue-Counter Update]}
Update the virtual queues assuming a precedence-relaxed system, \emph{i.e.},
\begin{eqnarray*}
	\tilde{Q}_e(t+1) \gets \bigg(\tilde{Q}_e(t)+A_e(t)-\mu_e(t)\bigg)^+, \hspace{5pt} \forall e \in E
\end{eqnarray*}
 \end{algorithmic}
 \end{algorithm} 
We next state the following theorem which proves stability of the physical queues under the \textbf{ENTO} policy:
\begin{framed}
 \begin{theorem} \label{physical_stability-theorem}
  Under the action of the \textbf{UMW} policy with \textbf{ENTO} packet scheduling, the physical queues are rate-stable \cite{neely2010stochastic} for any arrival vector $\bm{\lambda}\in \mathrm{int}(\overline{\bm{\Lambda}})$, i.e., 
  \begin{eqnarray*}
   \lim_{t \to \infty} \frac{\sum_{e \in E}Q_e(t)}{t} =0 , \hspace{10pt} \mathrm{w.p.} \hspace{1pt} 1
  \end{eqnarray*}
\end{theorem}
\end{framed}
\begin{proof}
This theorem is proved by extending the argument of Gamarnik \cite{Gamarnik} and combining it with the sample path loading condition in Eqn. \eqref{bounded_arrival}. See Appendix \ref{physical_stability-proof} for the detailed argument.  
\end{proof}
\vspace{-10pt}
As a direct consequence of Theorem \ref{physical_stability-theorem}, we have the main result of this paper:
\begin{framed}
\begin{theorem} \label{UMW-optimality}
 The \textbf{UMW} policy is throughput-optimal.
\end{theorem}
\end{framed}
\begin{proof}
For any class $c \in \mathcal{C}$, the number of packets $R^{(c)}(t)$, received by all nodes $ i \in \mathcal{D}^{(c)}$ may be bounded as follows:
\begin{eqnarray} \label{bd_eqn1}
  A^{(c)}(0,t) - \sum_{e \in E} {Q}_e(t) \stackrel{(*)}{\leq} R^{(c)}(t) \leq A^{(c)}(0,t),
\end{eqnarray}
where the lower-bound $(*)$ follows from the simple observation that if a packet $p$ of class $c$ has not reached all destination nodes $\mathcal{D}^{(c)}$, then at least one copy of it must be present in some physical queue. \\ 
Dividing both sides of Eqn. \eqref{bd_eqn1} by $t$, taking limits and using SLLN and Theorem \ref{physical_stability-theorem}, we conclude that w.p. $1$
\begin{eqnarray*}
 \lim_{t \to \infty} \frac{R^{(c)}(t)}{t} = \lambda^{(c)} 
\end{eqnarray*}
Hence from the definition \eqref{supporting_policy}, we conclude that \textbf{UMW} is throughput-optimal.

\end{proof}

\section{Distributed Implementation} \label{implementation-section}
The \textbf{UMW} policy in its original form, as given in Algorithm \ref{UMW_algo}, is centralized in nature. This is because the sources need to know the topology of the network and the current value of the virtual queues $\tilde{\bm{Q}}(t)$ to solve the shortest route and the Max-Weight problems at  steps (2) and (3) of the algorithm. Although the topology of the network may be obtained efficiently by topology discovery algorithms \cite{chandra2002mesh}, keeping track of the virtual queue evolution (Eqn. \eqref{queue-evolution}) is subtler. Note that, in the special case where all packets arrive only at a single source node, no information exchange is necessary and the virtual queue updates (Step 5) may be implemented at the source locally. In the general case with multiple sources, it is necessary to periodically exchange packet arrival information among the sources to implement Step 5 exactly. To circumvent this issue, a heuristic version of \textbf{UMW} policy (referred to as \textbf{UMW} (heuristic) in the following) may be used in practice where physical queue-lengths $\bm{Q}(t)$ are used as a surrogate for the virtual queue-lengths $\tilde{\bm{Q}}(t)$ for weight and cost computations in Algorithm \ref{UMW_algo}. Routing based on physical queue-lengths still requires the exchange of queue-length-information. However, this can be implemented efficiently using the standard distributed Bellman-Ford algorithm. The simulation results in section \ref{phy_q_sim} show that the heuristic policy works well in practice and its delay performance is substantially better than the virtual queue based optimal \textbf{UMW} policy in wireless networks.

\section{Numerical Simulation} \label{simulation-section}

\subsection{Delay Improvement Compared to the Back Pressure Policy - the Unicast Setting}

To empirically demonstrate superior delay performance of the \textbf{UMW} policy over the Back-Pressure policy in the unicast setting, we simulate both  policies in the wired network shown in Figure \ref{network-comp}. All links have a unit capacity. We consider two concurrent unicast sessions with source-destination pairs given by $(s_1=1,t_1=8)$ and $(s_2=5,t_2=2)$ respectively. It is easy to see that $\textsf{Max-Flow}(s_1\to t_1)=2$ and $\textsf{Max-Flow}(s_2\to t_2)=1$ and there exist mutually disjoint paths to achieve the optimal rate-pair $(\lambda_1,\lambda_2)=(2,1)$.
Assuming Poisson arrivals at the sources $s_1$ and $s_2$ with intensities $\lambda_1=2\rho$ and $\lambda_2=\rho$, $0 \leq \rho \leq 1$, where $\rho$   denotes the ``load factor", Figure \ref{multi-com} shows a plot of total average queue-lengths as a function of the load factor $\rho$ under the operation of the \textbf{BP}, \textbf{UMW} (optimal) and \textbf{UMW} (heuristic) policy.  


\begin{figure}
	\center
		\begin{overpic}[scale=0.36]{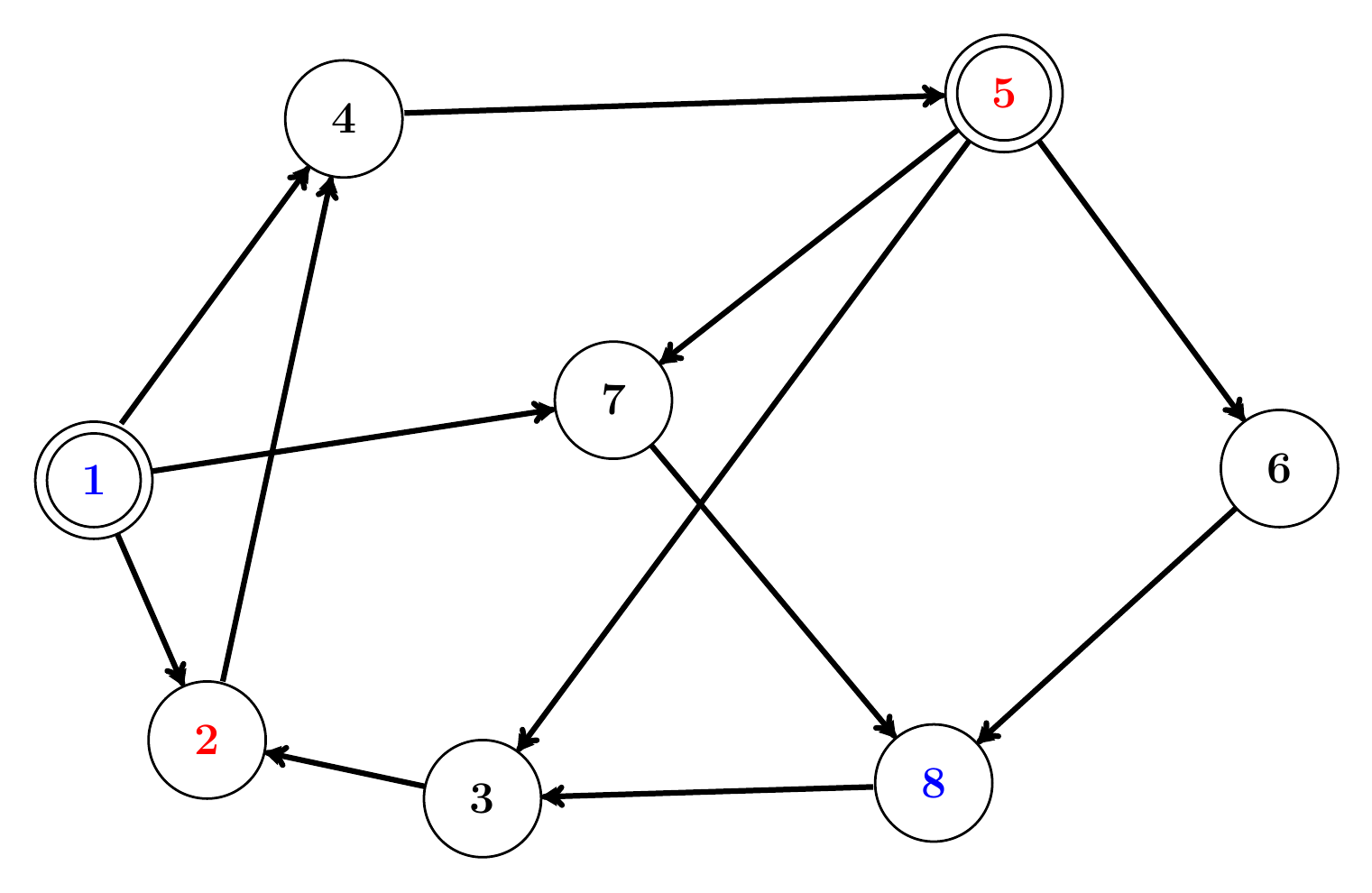}
		\put(-2,35){$\textcolor{blue}{s_1}$}
		\put(-6.5,29){$\textcolor{blue}{\Longrightarrow}$}
		\put(68,-1){$\textcolor{blue}{t_1}$}
		\put(75,65){$\textcolor{red}{s_2}$}
		\put(78,57){$\textcolor{red}{\Longleftarrow}$}
		\put(-13,29){$\textcolor{blue}{\lambda_1}$}
		\put(88,57){$\textcolor{red}{\lambda_2}$}
		\put(15,02){$\textcolor{red}{t_2}$}
		\put(71.5,1.5){$\textcolor{blue}{\searrow}$}
		\put(7,04){$\textcolor{red}{\swarrow}$}
		\end{overpic}
	\caption{\small{The wired network topology used for unicast simulation }}
	\label{network-comp}
\end{figure}

\begin{figure}
	\center
	\begin{overpic}[scale=0.24]{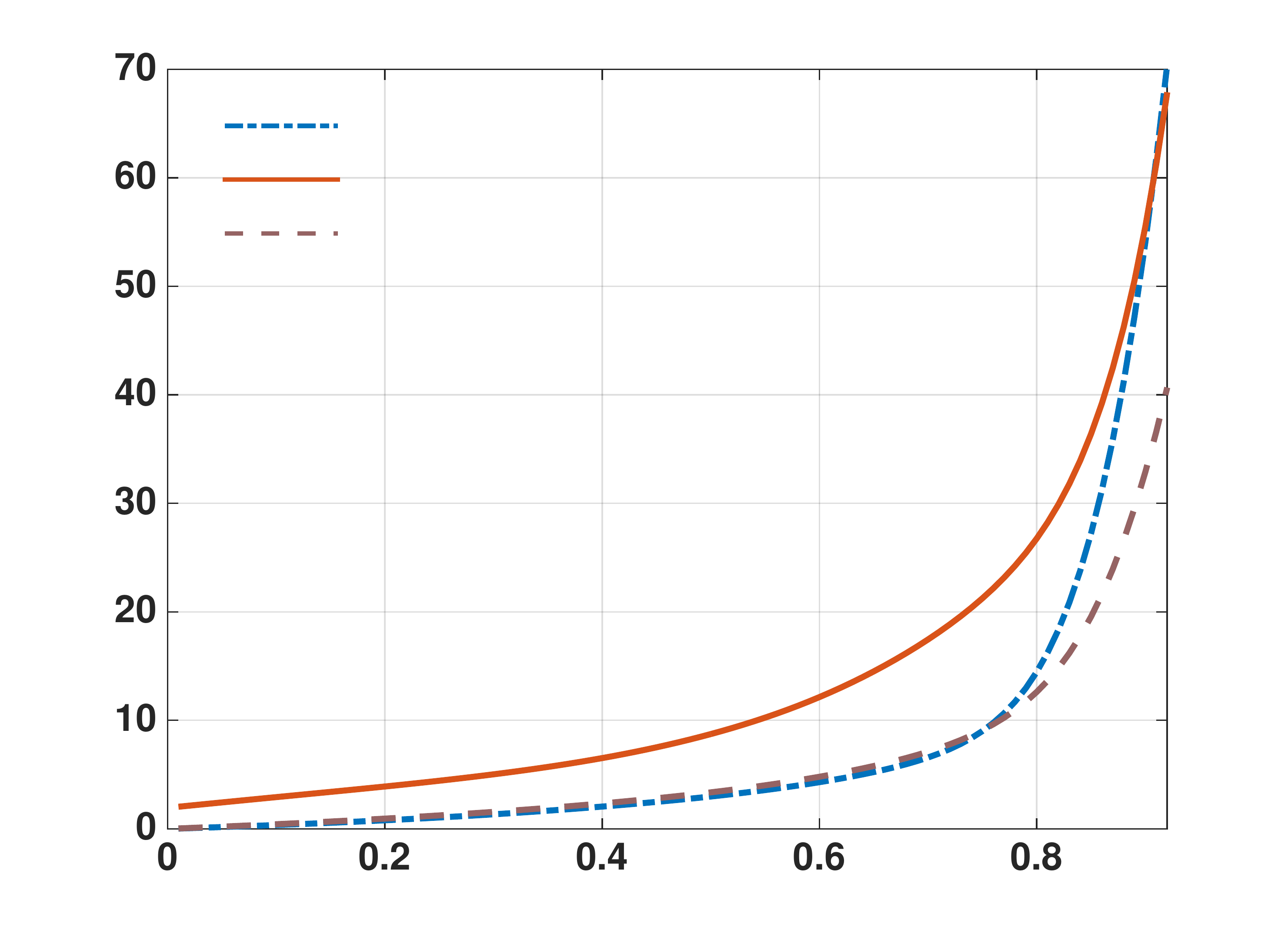}
	\put(52,1){$\rho$}
	\put(29,57){\footnotesize{\textbf{\textcolor{black}{BP}}}}
	\put(29,62){\footnotesize{\textbf{\textcolor{black}{UMW }}}(optimal)}
	\put(29,52){\footnotesize{\textbf{\textcolor{black}{UMW }}}(heuristic)}
	\put(2,20){\rotatebox{90}{\footnotesize{Avg. Queue Lengths}}}
	\end{overpic}
	\caption{\small{Comparing the delay performances of the \textbf{BP} and \textbf{UMW} (optimal and heuristic) policies in the unicast setting of Fig. \ref{network-comp}.}}
	\label{multi-com}
\end{figure}
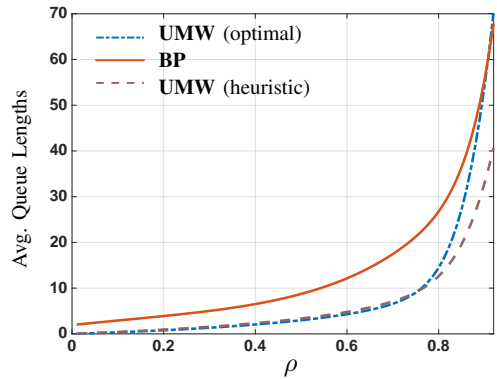

From the plot, we conclude that both the optimal and heuristic \textbf{UMW} policies outperforms the \textbf{BP} policy in terms of average queue-lengths, and hence (by Little's Law), end-to-end delay, especially in low-to-moderate load regime. The primary reason being, the \textbf{BP} policy, in principle, explores all possible paths to route packets to their destinations. In the low-load regime, the packets may also cycle in the network indefinitely, which increases their latency. The \textbf{UMW} policy, on the other hand, transmits all packets  along ``optimal" acyclic routes. This results in substantial reduction in end-to-end delay.

 
 \subsection{Using the Heuristic \textbf{UMW} policy for Improved Latency in the Wireless Networks - the Broadcast Setting}\label{phy_q_sim}
 Next, we empirically demonstrate that the heuristic \textbf{UMW} policy that uses physical queue-lengths $\bm{Q}(t)$ (instead of virtual queues $\bm{\tilde{Q}}(t)$ as in the optimal \textbf{UMW} policy) not only achieves the full broadcast capacity but yields better delay performance in this particular wireless network. As discussed earlier, the heuristic policy is practically easier to implement in a distributed fashion.
 We simulate a $3\times 3$ wireless grid network shown in Figure \ref{grid_topo}, with \emph{primary} interference constraints \cite{joo2009greedy}.
  The broadcast capacity of the network is known to be $\lambda^*=\frac{2}{5}$ \cite{Sinha:2016:TMB:2942358.2942390}. The ENTO policy is used for packet scheduling. The average queue-length is plotted in Figure \ref{delay_fig_UMW_BP1} as a function of the packet arrival rate $\lambda$ under the operation of the (a) \textbf{UMW} (optimal) and (b) \textbf{UMW} (heuristic) policies. 
   The plot  shows that the heuristic policy results in much smaller queue-lengths than the optimal policy. The reason being that physical queues capture the network congestion ``more accurately" for proper link activations. 
 
%

%
 
%
%
 \begin{figure}[h]
\centering
	\begin{minipage}{0.45\textwidth}
	\centering
		\begin{overpic}[scale=0.54]{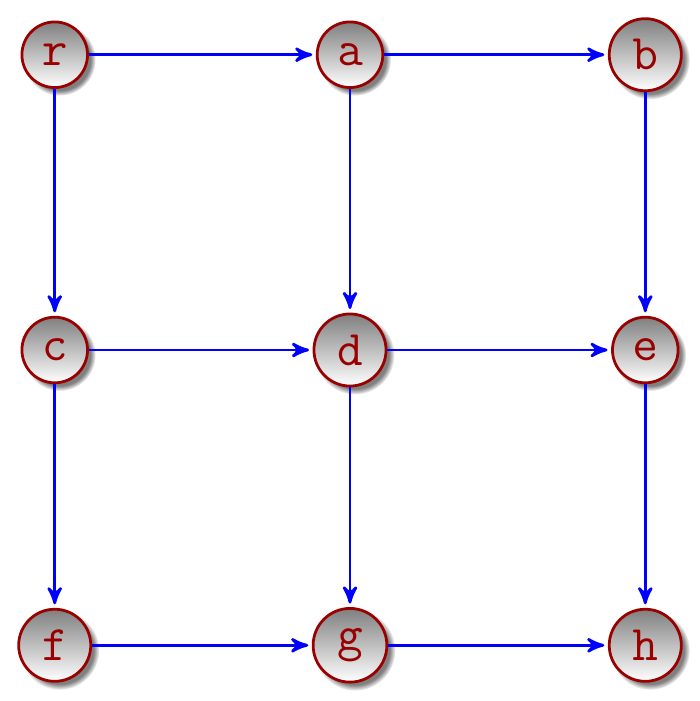}
		\put(-18,90){\textcolor{blue}{$\lambda$}}
		\put(-13,103){\small{Source}}
		\put(-12,90){$\textcolor{blue}{\Longrightarrow}$}
			\end{overpic}
			\caption{\small{The wireless topology used for broadcast simulation}}
			\label{grid_topo}
	\end{minipage}
\begin{minipage}{0.45\textwidth}
    \centering 
	\begin{overpic}[scale=0.23]{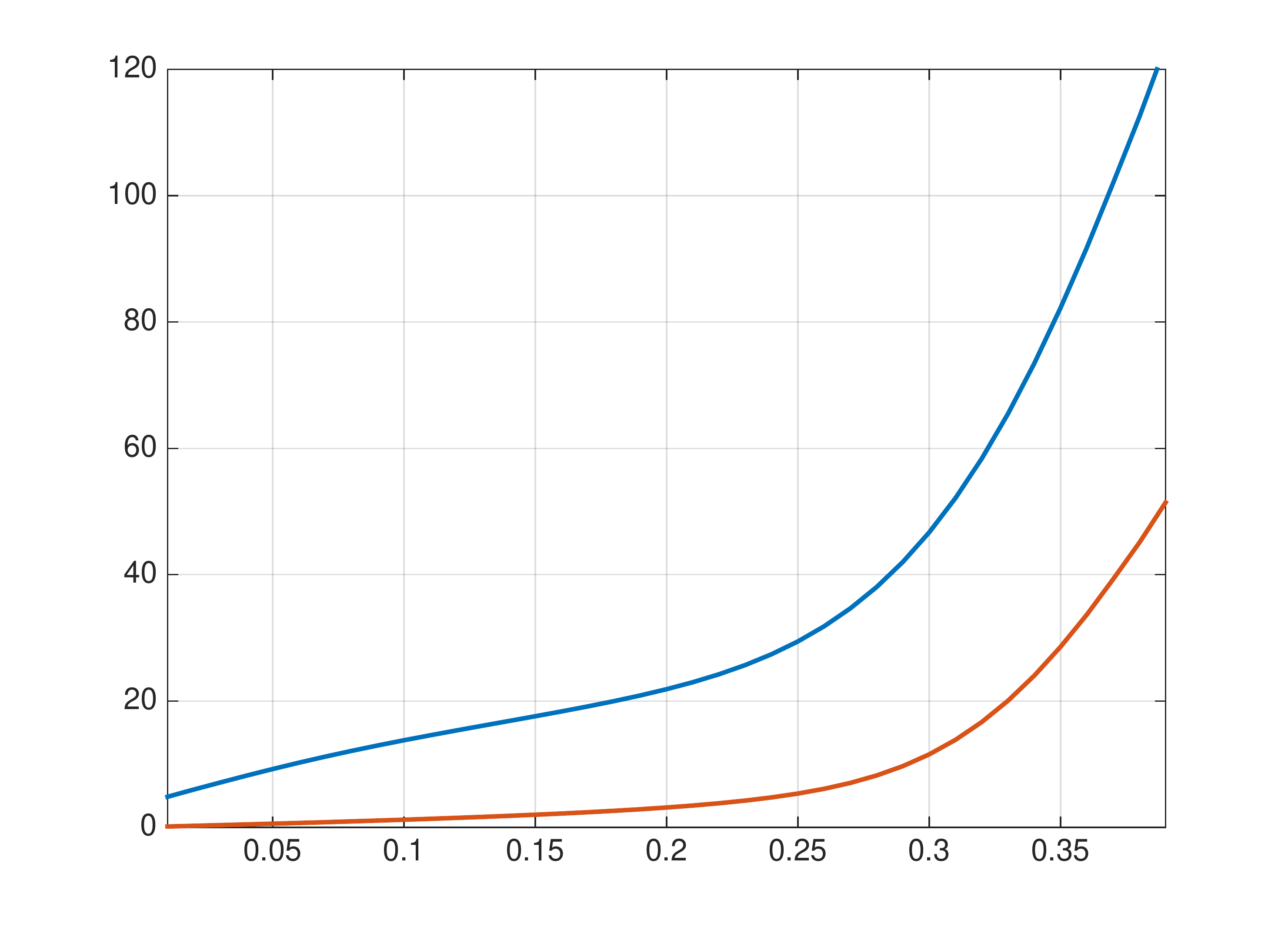}
		\put(40,-1){Arrival rate ($\lambda$)}
		\put(70,55.4){\footnotesize{\textbf{UMW} }}
		\put(62,55.4){\footnotesize{$(a)$}}
		\put(68,51){\footnotesize{(optimal)}}
		\put(71,25){\footnotesize{\textbf{UMW}}}
		\put(63,25){\footnotesize{(b)}}
		\put(69,20.5){\footnotesize{(heuristic)}}
		\put(1,20){\rotatebox{90}{\small{Avg. Queue lengths}}}
	\end{overpic}
	\caption{\small{Comparing Avg. Queue lengths as a function of arrival rate for the optimal (in blue) and the heuristic (in red) \textbf{UMW} Policy for the grid network in Figure \ref{grid_topo} in the \textbf{broadcast} setting.}}

	\label{grid_plot}
	\label{delay_fig_UMW_BP1}
	\end{minipage}
	\end{figure}


\section{Conclusion} \label{conclusion-section}
In this paper, we have proposed a new, efficient and throughput-optimal policy, named \textbf{U}niversal \textbf{M}ax-\textbf{W}eight (\textbf{UMW}), for the Generalized Network Flow problem. The \textbf{UMW} policy can simultaneously handle mix of Unicast, Broadcast, Multicast and Anycast traffic in arbitrary  networks and is empirically shown to have superior performance compared to the existing policies. The next step would be to investigate whether the \textbf{UMW} policy still retains its optimality when implemented with physical queue-lengths, instead of the virtual queue-lengths. An affirmative answer to this question would imply a more efficient implementation of the policy.

\bibliographystyle{IEEEtran}
\bibliography{MIT_broadcast_bibliography}

\clearpage
\section{Appendix}

\subsection{Proof of Converse of Theorem \ref{capacity-region-characterization}} \label{capacity-region-characterization-converse-proof}
\begin{proof}
Consider any admissible arrival rate vector $\bm{\lambda} \in \bm{\Lambda}(\mathcal{G}, \mathcal{C})$. By definition, there exists an admissible policy $\pi\in \Pi$ which supports the arrival vector $\bm{\lambda}$ in the sense of Eqn.  \eqref{supporting_policy}. Without any loss of generality, we may assume the policy $\pi$ to be stationary and the associated DTMC to be ergodic. Let $A^{(c)}_i(t)$ denote the total number of packets from class $c$ that have finished their routing along the route $T_i^{(c)} \in \mathcal{T}^{(c)}$ up to time $t$. Note that, each packet is routed along one admissible route only. Hence, if the total number of arrival to the source $s^{(c)}$ of class $c$ up to time $t$ is denoted by the random variable $A^{(c)}(t)$, we have 
\begin{eqnarray} \label{counting_pkts}
A^{(c)}(t) \stackrel{(a)}{\geq} \sum_{T^{(c)}_i \in \mathcal{T}^{(c)}}A^{(c)}_i(t) \stackrel{(b)}{=} R^{(c)}(t).
\end{eqnarray}
In the above, the inequality (a) follows from the observation that any packet $p$ which has finished its routing along some route $T_i^{(c)} \in \mathcal{T}^{(c)}$ by the time $t$, must have arrived at the source by the time $t$.  The equality (b) follows from the observation that any packet $p$ which has finished its routing by time $t$ along some route $T^{(c)}_i \in \mathcal{T}^{(c)}$, has reached all of the destination nodes $\mathcal{D}^{(c)}$  of class $c$ by time $t$ and vice versa. \\
Dividing both sides of equation \eqref{counting_pkts} by $t$ and taking limit as $t \to \infty$, we have $w.p. 1$
\begin{eqnarray*}
\lambda^{(c)}\stackrel{(d)}{=}\lim_{t \to \infty} \frac{A^{(c)}(t)}{t}  
&\geq &\liminf_{t \to \infty} \frac{1}{t} \sum_{T^{(c)}_i \in \mathcal{T}^{(c)}}A^{(c)}_i(t) \\
&=&  \liminf_{t \to \infty}\frac{R^{(c)}(t)}{t}\\
&\stackrel{(f)}{=}& \lambda^{(c)}, 
\end{eqnarray*} 
where equality (d) follows from the SLLN, 
and equality  (f) follows from the Definition \eqref{supporting_policy}. \\
From the above inequalities, we conclude that w.p. $1$
\begin{eqnarray}
 \lim_{t \to \infty} \frac{1}{t} \sum_{T^{(c)}_i \in \mathcal{T}^{(c)}}A^{(c)}_i(t)= \lambda^{(c)}, \hspace{10pt} \forall c \in \mathcal{C}
\end{eqnarray}
Now we use the fact that the policy $\pi$ is stationary and the associated DTMC is \emph{ergodic}. Thus the time-average limits exist and they are constant \emph{a.s.}. For all $T_i^{(c)} \in \mathcal{T}^{c}$ and $c \in \mathcal{C}$, define 
\begin{eqnarray}
 \lambda_i^{(c)} \stackrel{\mathrm{def}}{=} \lim_{t \to \infty} \frac{1}{t}A_i^{(c)}(t)
\end{eqnarray}
Hence, from the above, we get
\begin{eqnarray} \label{constr_eq1}
 \lambda^{(c)} = \sum_{T_i^{(c)} \in \mathcal{T}^{(c)}} \lambda_i^{(c)}
\end{eqnarray}
Now consider any edge $e \in E$ in the graph $\mathcal{G}$. Since the variable $A_i^{(c)}(t) $ denotes the total number of packets from class $c$, that have \emph{completely} traversed along the tree $T^{(c)}_i$, the following inequality holds good for any time $t$ 
\begin{eqnarray}
 \sum_{(i,c): e \in T_i^{(c)}, T_i^{(c)} \in \mathcal{T}^{(c)}} \hspace{-20pt}A_i^{(c)}(t) \leq \sum_{\tau =1}^{t} \mu_e(\tau),
\end{eqnarray}
where the left-hand side denotes a lower-bound on the number of packets that have crossed the edge $e$ and the right hand side denotes the amount of service that have been provided to edge $e$ up to time $t$ by the policy $\pi$.\\
Dividing both sides by $t$ and taking limits of both side, and noting that the limit on the left-hand side exists w.p. $1$, we have 
\begin{eqnarray} \label{constr_eq2}
 \sum_{(i,c): e \in T_i^{(c)}, T_i^{(c)} \in \mathcal{T}^{c}} \hspace{-20pt} \lambda_i^{(c)} \leq \overline{\mu}_e,
\end{eqnarray}
where $\overline{\bm{\mu}}= \lim_{t \to \infty} \frac{1}{t} \sum_{\tau=1}^{t} \bm{\mu}(\tau)$. Since $\bm{\mu}(\tau) \in \mathcal{M}, \forall \tau $ and the set $\text{conv}(\mathcal{M})$ is closed, we conclude that $\overline{\bm{\mu}} \in \text{conv}(\mathcal{M})$. Eqns. \eqref{constr_eq1} and \eqref{constr_eq2} concludes the proof of the theorem.


\end{proof}

\subsection{Proof of Theorem \ref{stability_theorem}} \label{stability_proof}
\begin{proof}
Consider an arrival rate vector $\bm{\lambda} \in \mathrm{int}(\overline{\bm{\Lambda})}$. Thus, from Eqns. \eqref{st1} and \eqref{st2}, it follows that there exists a scalar $\epsilon >0$ and a vector $\bm{\mu} \in \mathrm{conv}(\mathcal{M})$, such that we can decompose the total arrival for each class $c \in \mathcal{C}$ into a finite number of routes, such that 
\begin{eqnarray} \label{diff}
\lambda_e \stackrel{(\text{def.})}{=} \hspace{-20pt}\sum_{(i,c): e \in T^{(c)}_i\hspace{-2pt}, T_i^{(c)} \hspace{-2pt}\in \mathcal{T}^{(c)}} \hspace{-20pt} \lambda^{(c)}_i \leq \mu_e - \epsilon, \hspace{10pt} \forall e \in E
\end{eqnarray}
By Caratheodory's theorem \cite{matouvsek2002lectures}, we can write 
\begin{eqnarray}
\bm{\mu}= \sum_{i=1}^{m+1} p_i \bm{s}_i, 
\end{eqnarray}
for some activation vectors $\bm{s}_i \in \mathcal{M}, \forall i$ and some probability distribution $\bm{p}$. \\
Now consider the following auxiliary stationary randomized routing and link activation policy \textbf{RAND} $\in \Pi$ for the virtual queue system $\{\tilde{\bm{Q}}(t)\}$, which will be useful in our proof. The randomized policy $\textbf{RAND}$ randomly selects the activation vector $\bm{s}_j$ with probability $p_j, j=1,2,\ldots, m+1$ and routes the incoming packet of class $c$ along the route $T^{(c)}_i \in \mathcal{T}^{(c)}$, with probability $\frac{\lambda_i^{(c)}}{\lambda^{(c)}}, \hspace{1pt}\forall i, c $. Hence the total expected arrival rate to the virtual queue $\tilde{Q}_e$ at time slot $t$, due to the action of the stationary randomized policy \textbf{RAND}  is given by 
\begin{eqnarray} \label{tot_arr}
\mathbb{E}A_e^{\mathbf{RAND}}(t) =\lambda_e= \hspace{-20pt}\sum_{(i,c): e \in T^{(c)}_i\hspace{-2pt}, T_i^{(c)} \hspace{-2pt}\in \mathcal{T}^{(c)}} \hspace{-20pt} \lambda^{(c)}_i , \hspace{10pt} \forall e \in E 
\end{eqnarray}
and the expected total service rate to the virtual server for the queue $\tilde{Q}_e$ is given by 
\begin{eqnarray} \label{tot_ser}
\mathbb{E}\mu_e^{\mathbf{RAND}}(t)=\sum_{i=1}^{m+1}p_i\bm{s}_i(e)=\mu_e
\end{eqnarray}
Since our Max-Weight policy, \textbf{UMW}, maximizes  the RHS of the drift expression in Eqn. \eqref{drift_expr} from the set of all feasible policies $\Pi$, we can write 
\begin{eqnarray*}
\Delta^{\mathbf{UMW}}(t) &\leq& B + 2 \sum_{e \in E}\tilde{Q}_e(t)\mathbb{E}\big(A_e^\mathbf{RAND}(t)|\bm{\tilde{Q}}(t)\big)\nonumber \\
	&-& 2 \sum_{e \in E}\tilde{Q}_e(t)\mathbb{E}\big(\mu_e^\mathbf{RAND}(t)|\bm{\tilde{Q}}(t)\big)\\
	\end{eqnarray*}
\begin{eqnarray*}
&\stackrel{(a)}{=}& B+ 2\sum_{e \in E} \tilde{Q}_e(t)\bigg(\mathbb{E}A_e^\mathbf{RAND}(t)-\mathbb{E}\mu_e^\mathbf{RAND}(t) \bigg)\\
&\stackrel{(b)}{=}& B+2\sum_{e \in E} \tilde{Q}_e(t)\big(\lambda_e -\mu_e \big)\\
&\stackrel{(c)}{\leq}& B - 2\epsilon \sum_{e \in E} \tilde{Q}_e(t),
\end{eqnarray*}
where (a) follows from the fact that the randomized policy \textbf{RAND} is memoryless and hence, independent of the virtual queues $\bm{\tilde{Q}}(t)$, (b) follows from Eqns. \eqref{tot_arr} and \eqref{tot_ser} and finally (c) follows from Eqn. \eqref{diff}.\\
Taking expectation of both sides w.r.t. the virtual queue-lengths $\bm{\tilde{Q}}(t)$, we bound the expected drift at slot $t$ as 
\begin{eqnarray} \label{drift_t}
\mathbb{E}L\big(\tilde{\bm{Q}}(t+1)\big) - \mathbb{E}L\big(\tilde{\bm{Q}}(t)\big) \leq B - 2\epsilon \sum_{e \in E} \mathbb{E}(\tilde{Q}_e(t))
\end{eqnarray}
Summing Eqn. \eqref{drift_t} from $t=0$ to $T-1$ and remembering that $L(\bm{Q}(T)) \geq 0$ and $L(\bm{\tilde{Q}}(0))=0$, we conclude that 
\begin{eqnarray}
\frac{1}{T}\sum_{t=0}^{T-1}\sum_{e \in E} \mathbb{E}(\tilde{Q}_e(t)) \leq \frac{B}{2 \epsilon}
\end{eqnarray}
Taking $\limsup$ of both sides proves the claim. 
\end{proof}

\subsection{Proof of the Skorokhod Map Representation in Eqn. \eqref{skorokhod}} \label{skorokhod_proof}
\begin{proof}
	From the dynamics of the virtual queues in Eqn. \eqref{queue-evolution}, we have for any $t \geq 1$ 
	\begin{eqnarray} \label{recursion}
		\tilde{Q}_e(t) \geq \tilde{Q}_e(t-1) + A_e(t-1)-\mu_e(t-1).
	\end{eqnarray}

Iterating \eqref{recursion} $\tau$ times $1\leq \tau \leq t$, we obtain 
\begin{eqnarray*}
	\tilde{Q}_e(t) \geq \tilde{Q}_e(t-\tau) + A_e(t-\tau, t)-S_e(t-\tau, t),
\end{eqnarray*}

where $A_e(t_1,t_2)=\sum_{\tau=t_1}^{t_2-1} A_e(\tau) $ and $S_e(t_1,t_2)=\sum_{\tau=t_1}^{t_2-1} \mu_e(\tau)$, as defined before. Since each of the virtual-queue components are non-negative at all times (viz. \eqref{queue-evolution}), we have $\tilde{Q}_e(t-\tau)\geq 0$. Thus, \begin{eqnarray*}
	\tilde{Q}_e(t) \geq A_e(t-\tau, t)-S_e(t-\tau, t).
\end{eqnarray*}
Since the above holds for any time $1 \leq \tau \leq t$ and the queues are always non-negative, we obtain
\begin{eqnarray} \label{lb_skorokhod}
	\tilde{Q}_e(t) \geq \bigg(\sup_{1 \leq \tau \leq t}\big(A_e(t-\tau, t)-S_e(t-\tau, t)\big)\bigg)_+
\end{eqnarray}

To show that Eqn. \eqref{lb_skorokhod} holds with equality, we consider two cases. \\\\
\underline{\textbf{Case I: $\tilde{Q}_e(t)=0$} } \\
Since the RHS of Eqn. \eqref{lb_skorokhod}
 is non-negative, we immediately obtain equality throughout in Eqn \eqref{lb_skorokhod}. \\\\
 \underline{\textbf{Case II: $\tilde{Q}_e(t)>0$} }
Consider the \emph{latest time} $t-\tau', 1\leq \tau'\leq t$, prior to $t$, at which $\tilde{Q}_e(t-\tau')=0$. Such a time $t-\tau'$ exists because we assumed the system to start with empty queues at time $t=0$. Hence $Q_e(z) > 0$ throughout the time interval $z\in [t-\tau'+1, t]$. As a result, in this time interval the system dynamics for the virtual-queues \eqref{queue-evolution} takes the following form 
\begin{eqnarray*}
\tilde{Q}_e(z) = \tilde{Q}_e(z-1) + A_e(z-1)-\mu_e(z-1), 
\end{eqnarray*}
Iterating the above recursion in the interval $z \in [t-\tau'+1, t]$, we obtain
\begin{eqnarray} \label{ach_skorokhod}
\tilde{Q}_e(t)= A_e(t-\tau',t)-S_e(t-\tau',t)
\end{eqnarray}
We conclude the proof upon combining Eqns. \eqref{lb_skorokhod} and \eqref{ach_skorokhod}.
\end{proof}

\subsection{Proof of Lemma \ref{rate_stability}}\label{rate_stability_proof}

\begin{proof}
	We will establish this result by appealing to the \emph{Strong Stability Theorem} (Theorem 2.8) of \cite{neely2010stochastic}. For this, we first consider an associated system $\{\hat{\bm{Q}}(t)\}_{t \geq 0}$ with a slightly different queueing recursion, as considered in \cite{neely2010stochastic} (Eqn. 2.1, pp-15). For a given sequence $\{\bm{A}(t), \bm{\mu}(t)\}_{t \geq 0}$, define the following recursion for all $e \in E$,
	\begin{eqnarray}\label{assoc_eqn}
		\hat{Q}_e(t+1)&=& (\hat{Q}_e(t)-\mu_e(t))_+ + A_e(t),\\
		\hat{Q}_e(0)&=&0.\nonumber
	\end{eqnarray}
	Recall the dynamics of the virtual queues (Eqn. \eqref{queue-evolution}):
	\begin{eqnarray}\label{lind_eqn}
		\tilde{Q}_e(t+1)&=&(\tilde{Q}_e(t)+ A_e(t)- \mu_e(t))^+,\\
		\tilde{Q}_e(0)&=&0.\nonumber
	\end{eqnarray}
	We next prove the following proposition:  
	\begin{proposition} \label{comp_propo}
	For all $e \in E$
	\begin{eqnarray*} 
	A_{\max}+ \tilde{Q}_e(t) \stackrel{(*)}{\geq} \hat{Q}_e(t) \stackrel{(**)}{\geq} \tilde{Q}_e(t), \hspace{10pt} \forall t\geq 0.	
	\end{eqnarray*}
\end{proposition}
	\begin{proof}
	We first prove the second inequality (**) by inducting on time.\\ 
	\underline{\textbf{Base Step $t=0$}}:\\
	Holds with equality since $\hat{Q}_e(0)=\tilde{Q}_e(0)=0$. 
	\underline{\textbf{Induction Step}}:\\
	Assume that $\hat{Q}_e(t) \geq \tilde{Q}_e(t)$ for some $t \geq 0$. From the dynamics \eqref{assoc_eqn}, we can write 
	\begin{eqnarray*}
		\hat{Q}_e(t+1)&=&\max\big( \hat{Q}_e(t)-\mu_e(t)+A_e(t), A_e(t)\big)\\
		&\stackrel{(a)}{\geq} & \max\big( \hat{Q}_e(t)-\mu_e(t)+A_e(t), 0\big)\\
		& \stackrel{(b)}{\geq}&\max \big(\tilde{Q}_e(t)-\mu_e(t)+A_e(t),0\big)\\
		&\stackrel{(c)}{=}&\tilde{Q}_e(t+1), 
	\end{eqnarray*}
	where Eqn. (a) follows from the fact that $A_e(t) \geq 0$, Eqn. (b) follows from the induction assumption and Eqn. (c) follows from the dynamics \eqref{lind_eqn}. This completes the induction step and the proof of the second inequality (**) of the proposition. Proof of the first inequality (*) may also be carried out similarly. 
	\end{proof}
	Taking expectation throughout the first inequality (*) of Proposition \ref{comp_propo} for any $e \in E$, we have for each $t\geq 0$
	\begin{eqnarray*}
		\mathbb{E}(\hat{Q}_e(t))\leq  \mathbb{E}(\tilde{Q}_e(t)) + A_{\max}
	\end{eqnarray*} 
	Thus,
	\begin{eqnarray*}
		\limsup_{T\to \infty} \frac{1}{T}\sum_{t=0}^{T-1}\mathbb{E}(\hat{Q}_e(t))&\leq& \limsup_{T\to \infty} \frac{1}{T}\sum_{t=0}^{T-1}\mathbb{E}(\tilde{Q}_e(t)) + A_{\max}\\
		&\stackrel{(a)}{<} & \infty,
	\end{eqnarray*}
	where (a) follows from the strong stability of the virtual queues under \textbf{UMW}. This shows that, the associated queue-process $\{\hat{\bm{Q}}(t)\}_{t\geq 0}$ is also strongly stable under \textbf{UMW}.\\ 
	 Since the total external arrival $A(t)=\sum_{e}A_e(t)$ at slot $t$ is assumed to be bounded w.p. $1$, applying Theorem 2.8, part (b) of \cite{neely2010stochastic}, we conclude that for any $e \in E$ 
	\begin{eqnarray*}
		\lim_{t \to \infty} \frac{\hat{Q}_e(t)}{t}=0, \hspace{10pt} \mathrm{w.p.} 1
	\end{eqnarray*}
	Using the second inequality (**) of Proposition \ref{comp_propo} and the non-negativity of the virtual queues, we conclude that for any $e \in E$ 
	\begin{eqnarray*}
		\lim_{t \to \infty} \frac{\tilde{Q}_e(t)}{t}=0, \hspace{10pt} \mathrm{w.p.} 1
\end{eqnarray*}
Finally, using the union bound we conclude that 
	\begin{eqnarray*}
		\lim_{t \to \infty} \frac{\tilde{Q}_e(t)}{t}=0, \hspace{5pt} \forall e \in E \hspace{10pt} \mathrm{w.p.} 1
\end{eqnarray*}
\end{proof}

\subsection{Proof of Theorem \ref{physical_stability-theorem} } \label{physical_stability-proof}
Throughout this proof, we will fix a sample point $\omega \in \Omega$, giving rise to a sample path satisfying the condition \eqref{bounded_arrival}. All random processes \footnote{Recall that, a discrete-time integer-valued random process $\bm{X}(\omega;t)$ is a measurable map from the sample space $\Omega$ to the set of all integer-sequences $\mathbb{Z}^{\infty}$ \cite{durrett2010probability}, i.e., $\bm{X}: \Omega \to \mathbb{Z}^{\infty}$. } will be evaluated at this sample path. For the sake of notational simplicity, we will drop the argument $\omega$ for evaluating any random variable $\bm{X}$ at the sample point $\omega$, e.g., the deterministic sample-path $X(\omega, t)$ will be simply denoted by $X(t)$. We now establish a simple analytical result which will be useful in the main proof of the theorem:
\begin{framed}
\begin{lemma} \label{seq_lemma}
 Consider a non-negative function $\{F(t), t \geq 1\}$ defined on the set of natural numbers, such that $F(t)=o(t)$ . Define $M(t)= \sup_{0\leq \tau \leq t} F(\tau)$.  Then\\
 1. $M(t)$ is non-decreasing in $t$.\\
 2. $M(t)=o(t)$
\end{lemma}
\end{framed}
\begin{proof}
That $M(t)$ is non-decreasing follows directly from the definition of $M(t)=\sup_{0\leq \tau \leq t}F(t)$. We now prove the claim (2).\\
\textbf{Case I: The function $F( t)$ is bounded}\\
In this case, the function $M(t)$ is also bounded and the claim follows immediately.\\
\textbf{Case II: The function $F(t)$ is unbounded}\\
Define the subsequence $\{r_k\}_{k\geq 1}$, corresponding to the time of maximums of the function $M(t)$ up to time $t$. Formally the sequence $\{r_k\}_{k\geq 1}$ is defined recursively as follows,
\begin{eqnarray}
r_1&=&1\\
r_{k}&=& \{ \min t > r_{k-1} : F(t) > \max_{\tau \leq t-1}  F( \tau ) \}
\end{eqnarray}
Since the function $F( t)$ is assumed to be unbounded, we have $r_k \to \infty$ as $k \to \infty$. In the literature \cite{glick1978breaking}, the sequence $\{r_k\}$ is also known as the sequence of \emph{records} of the function $F( t)$. With this definition, for any $t \geq 1$ and for  $r_k \leq t$ corresponding to the latest record up to time $t$, we readily have
\begin{eqnarray}
M(t) = F(r_k)
\end{eqnarray}
Hence, 
\begin{eqnarray}
\frac{M(t)}{t} = \frac{F( r_k)}{t} \stackrel{(a)}{\leq} \frac{F( r_k)}{r_k},
\end{eqnarray}
where Eqn. (a) follows from the fact that $r_k \leq t$. Thus for any sequence of natural numbers $\{t_i\}_{1}^{\infty}$, we have a corresponding sequence  $\{r_{k_i}\}_{i=1}^{\infty}$ such that for each $i$, we have
\begin{eqnarray*}
\frac{M(t_i)}{t_i} = \frac{F(r_{k_i})}{t} \stackrel{(a)}{\leq} \frac{F( r_{k_i})}{r_{k_i}}
\end{eqnarray*}
This implies, 
\begin{eqnarray} \label{record}
\limsup_{t \to \infty} \frac{M( t)}{t} \leq \limsup_{t \to \infty} \frac{F( t)}{t}\stackrel{(b)}{=} 0, 
\end{eqnarray}
where Eqn (b) follows from our hypothesis on the function $F(t)$. Also since  $M(t)\geq F( t)$, from Eqn. \eqref{record} we conclude that 
\begin{eqnarray}
\lim_{t \to \infty} \frac{M( t)}{t} =0 
\end{eqnarray}

\end{proof}

As a direct consequence of Lemma \ref{seq_lemma} and the property of the sample-point $\omega$ under consideration, we have:
\begin{framed}
\begin{eqnarray} \label{arr_cond}
 A_e(t_0,t) \leq S_e(t_0,t) + M(t), \hspace{10pt} \forall e \in E, \hspace{3pt}\forall t_0 \leq t
\end{eqnarray}
\end{framed}
for some non-decreasing non-negative function $M(t)=o(t)$. Equipped with Eqn. \eqref{arr_cond}, we return to the proof of the Theorem \ref{physical_stability-theorem}. \\
\begin{framed}
\begin{proposition}
 \textsf{ENTO} is rate-stable.
 \end{proposition}
\end{framed}
\begin{proof}
We generalize the argument by Gamarnik \cite{Gamarnik} to prove the proposition. We remind the reader that we are analyzing the time-evolution of a fixed sample point $\omega \in \Omega$, which satisfies Eqn. \eqref{arr_cond}. \\
Let $R_e(0)$ denote the total number of packets waiting to cross the edge $e$ at time $t=0$. Also, let $R_k(t)$ denote the total number of packets at time $t$, which are \emph{exactly} $k$ hops away from their respective sources. Such packets will be called ``layer $k$" packets in the sequel. If a packet is duplicated along its assigned route $T$ (which is, in general, a tree), each copy of the packet is counted separately in the variable $R_k(t)$, \emph{i.e.,}
\begin{eqnarray} \label{r-k}
R_k(t)=\sum_{T \in \mathcal{T}} R_{(e^T_k, T)}(t), 
\end{eqnarray} 
where the variable $R_{(e,T)}(t)$ denotes the number of packets following the routing tree $T$, that are waiting to cross the edge $e \in T$ at time $t$. The edge $e^T_k$ is an edge located $k$\textsuperscript{th} hop away from the source in the tree $T$. If there are more than one such edge (because the tree $T$ has more than one branch), we include all these edges in the summation \eqref{r-k}.
We show by induction that $R_k(t)$ is \emph{almost surely} bounded by a function, which is $o(t)$.\\
\textbf{Base Step} $k=0$: Fix an edge $e$ and time $t$. Let $t_0\leq t$ be the largest time at which no packets of layer $0$ (packets which have not crossed any edge yet) were waiting to cross $e$. If no such time exists, set $t_0=0$. Hence, the total number of layer $0$ packets waiting to cross the edge $e$ at time $t_0$ is at most $Q_e( 0)$. During the time interval $[t_0,t]$, as a consequence of the \textbf{UMW} control policy \eqref{arr_cond}, at most $S_e(t_0,t)+M(t)$ external packets have been admitted to the network, that want to cross the edge $e$ in future. Also, by the choice of the time $t_0$, the edge $e$ was always having packets to transmit during the entire time interval $[t_0,t]$. Since \textsf{ENTO} scheduling policy is followed, layer $0$ packets have priority over all other packets.  Hence, it follows that the total number of packets at the edge $e$ at time $t$ satisfies 
\begin{eqnarray}
\sum_{T: e \in e^{T}_0}R_{(e,T)} (t) &\leq& R_e(0) + S_e(t_0,t)+M(t) - S_e(t_0,t) \nonumber \\
&\leq&  R_e(0)+ M(t) 
\end{eqnarray}
As a result, we have $R_0(t) \leq \sum_e R_e(0) + |E|M(t)$, for all $t$. Let $B_0(t)\stackrel{\text{def}}{=}\sum_e R_e(0) + |E|M(t)$. Since $M(t)=o(t)$, we have $B_0(t)=o(t)$. Note that, since $M(t)$ is monotonically non-decreasing by definition, so is $B_0(t)$. 

\textbf{Induction Step:}
Suppose that, for some monotonically non-decreasing functions $B_j(t)=o(t), j=0,1,2,\ldots, k-1$, we have $ R_j(t) \leq B_j(t)$, for all time $t$. We next show that $ R_k(t) \leq B_k(t)$ for all $t$, where $B_k(t)=o(t)$. \\
Again, fix an edge $e$ and an arbitrary time $t$. Let $t_0\leq t$ denote the largest time before $t$, such that there were no layer $k$ packets waiting to cross the edge $e$. Set $t_0=0$ if no such time exists. Hence the edge $e$ was always having packets to transmit during the time interval $[t_0,t]$ (packets in layer $k$ or lower). The layer $k$ packets that wait to cross edge $e$ at time $t$ are composed only of a subset of packets which were in layers $0\leq j \leq k-1$ at time $t_0$ or packets that arrived during the time interval $[t_0, t]$ and have edge $e$ as \emph{one of their $k$\textsuperscript{th} edge} on the route followed. By our induction assumption, the first group of packets has a size bounded by $\sum_{j=0}^{k-1} B_j(t_0)\leq \sum_{j=0}^{k-1}B_j(t)$, where we have used the fact (from our previous induction step) that the functions $B_j(\cdot)$'s are monotonically non-decreasing. The size of the second group of packets is given by $\sum_{T: e \in e^{T}_k}A_T(t_0,t)$. We next estimate the number of layer $k$ packets that crossed the edge $e$ during the time interval $[t_0,t]$. Since \textsf{ENTO} policy is used, layer $k$ packets were not processed only when there were packets in layers up to $k-1$ that wanted to cross $e$. The number of such packets is bounded by $\sum_{j=0}^{k-1}B_j(t_0)\leq \sum_{j=0}^{k-1}B_j(t)$, which denotes the total possible number of packets in layers up to $k-1$ at time $t_0$, plus $\sum_{j=0}^{k-1} \sum_{T: e \in e^T_j}A_T(t_0,t)$, which is the number of new packets that arrived in the interval $[t_0,t]$ and intend to cross the edge $e$ within first $k-1$ hops. Thus, we conclude that at least 
\begin{eqnarray}
\max \bigg\{0, S_e(t_0,t) - \sum_{j=0}^{k-1}B_j(t) - \sum_{j=0}^{k-1}\sum_{T: e \in e_j^{T}} A_T(t_0,t) \bigg \}
\end{eqnarray}  
packets of layer $k$ crossed $e$ during the time interval $[t_0,t]$. Hence, 
\begin{eqnarray*}
&\sum_{T: e \in e^{T}_k}\hspace{-10pt}& R_{(e,T)}(t) \leq \sum_{j=0}^{k-1} B_j(t) + \sum_{T: e \in e^{T}_k}A_T(t_0,t) \\
&-& \big(S_e(t_0,t) - \sum_{j=0}^{k-1}B_j(t) - \sum_{j=0}^{k-1}\sum_{T: e \in e_j^{T}} A_T(t_0,t)\big)\\
&=& 2\sum_{j=0}^{k-1}B_j(t) + \sum_{j=0}^{k}\sum_{T: e \in e_j^{T}} A_T(t_0,t) - S_e(t_0,t) \\
& \stackrel{(a)}{\leq} & 2\sum_{j=0}^{k-1}B_j(t) + M(t),
\end{eqnarray*}
where Eqn. (a) follows from the arrival condition \eqref{arr_cond}. Hence the total number of layer $k$ packets at time $t$ is bounded by 
\begin{eqnarray}
R_k(t) \leq 2|E|\sum_{j=0}^{k-1}B_j(t) + M(t)|E|
\end{eqnarray}
Define $B_k(t)$ to be the RHS of the above equation, i.e. 
\begin{eqnarray}\label{bk}
B_k(t) \stackrel{(\text{def})}{=}  2|E|\sum_{j=0}^{k-1}B_j(t) + M(t)|E|
\end{eqnarray} 
Using our induction assumption and Eqn. \eqref{bk}, we conclude that $B_k(t)=o(t)$ and it is monotonically non-decreasing. This completes the induction step. \\
To conclude the proof of the proposition, notice that total size of the physical queues at time $t$ may be alternatively written as 
\begin{eqnarray}\label{sumQ}
\sum_{e \in E}{Q}_e(t)=\sum_{k=1}^{n-1} R_k(t)
\end{eqnarray}
Since the previous inductive argument shows that for all $k$, we have $R_k(t) \leq B_k(t)$ where $B_k(t)=o(t)$ \emph{a.s.}, we conclude that
\begin{eqnarray}
\lim_{t \to \infty} \frac{\sum_{e \in E}{Q}_e(t)}{t}=0, \hspace{15pt} \text{w.p. } 1,
\end{eqnarray}
This implies that the physical queues are rate stable \cite{neely2010stochastic}, jointly under the operation of \textbf{UMW} and \textsf{ENTO}. 
\end{proof}

\end{document}